\documentclass[table]{amsart} 
\usepackage{float}
\usepackage[margin=1.0in]{geometry}
\usepackage{amscd,amsmath,amsxtra,amsthm,amssymb,stmaryrd,xr,mathrsfs,mathtools,enumerate,commath, comment, enumitem, graphicx}
\usepackage{amsfonts, amssymb, amsthm, amsmath, braket, hyperref, mathtools, comment, mathrsfs, enumitem, cleveref}
\usepackage{stmaryrd}
\usepackage{orcidlink}
\usepackage{multirow}
\usepackage{xcolor}
\usepackage{listings}
\usepackage{commath}
\usepackage{comment}
\usepackage{multicol}
\usepackage{tikz-cd}
\usepackage{tkz-graph}
\usepackage{colonequals}
\usepackage{hyperref}
\usepackage{graphicx}
\usepackage{bigstrut}
\usepackage{amsmath, amssymb}
\usepackage{tikz}
\usepackage{empheq}
\usepackage{enumitem}
\usepackage{longtable, multirow, array}
\usepackage{placeins}

\newcommand{\Mod}[1]{\ (\mathrm{mod}\ #1)} %For mod in paranthesis

\newcommand{\sfrak}{\mathfrak{s}}

\newcommand{\scrE}{\mathscr{E}}

\usepackage{longtable} % for 'longtable' environment
\usepackage{pdflscape} % for 'landscape' environment
\usepackage{booktabs}
\usepackage{hyperref}
\definecolor{vegasgold}{rgb}{0.77, 0.7, 0.35}
\definecolor{darkgoldenrod}{rgb}{0.72, 0.53, 0.04}
\definecolor{gold(metallic)}{rgb}{0.83, 0.69, 0.22}
\hypersetup{
 colorlinks=true,
 linkcolor=darkgoldenrod,
 filecolor=brown,      
 urlcolor=gold(metallic),
 citecolor=darkgoldenrod,
 }

\usepackage[all,cmtip]{xy}
\DeclareFontFamily{U}{wncy}{}
\DeclareFontShape{U}{wncy}{m}{n}{<->wncyr10}{}
\DeclareSymbolFont{mcy}{U}{wncy}{m}{n}
\DeclareMathSymbol{\Sh}{\mathord}{mcy}{"58}
\usepackage[T2A,T1]{fontenc}
\usepackage[OT2,T1]{fontenc}

\usepackage{tikz}
\usetikzlibrary{shapes.geometric}
\usetikzlibrary{decorations.markings}
\tikzset{every loop/.style={min distance=10mm,looseness=10}}
\tikzstyle{vertex}=[auto=left,circle,minimum size=1pt,inner sep=0pt]

\newtheorem{theorem}{Theorem}[section]
\newtheorem{lemma}[theorem]{Lemma}

\newtheorem*{theorem*}{Theorem}
\newtheorem*{ass*}{Assumption}
\newtheorem{definition}[theorem]{Definition}
\newtheorem{corollary}[theorem]{Corollary}
\newtheorem{remark}[theorem]{Remark}
\newtheorem{example}[theorem]{Example}

\newtheorem{proposition}[theorem]{Proposition}

\renewcommand{\a}{\alpha}
\newcommand{\ao}{\alpha_0}
\renewcommand{\b}{\beta}
\newcommand{\bo}{\beta_0}
\newcommand{\rad}{\operatorname{rad}}

\newcommand{\Z}{\mathbb{Z}}

\newcommand{\Q}{\mathbb{Q}}
\newcommand{\F}{\mathbb{F}}
\newcommand{\D}{\mathfrak{D}}

\newcommand{\calC}{\mathcal{C}}
\newcommand{\calE}{\mathcal{E}}

\DeclareMathOperator{\Res}{Res}

\numberwithin{equation}{section}

\definecolor{codegreen}{rgb}{0,0.6,0}
\definecolor{codegray}{rgb}{0.5,0.5,0.5}
\definecolor{codepurple}{rgb}{0.58,0,0.82}
\definecolor{backcolour}{rgb}{0.95,0.95,0.92}

\lstdefinestyle{sageStyle}{
    backgroundcolor=\color{backcolour},   
    commentstyle=\color{codegreen},
    keywordstyle=\color{magenta},
    numberstyle=\tiny\color{codegray},
    stringstyle=\color{codepurple},
    basicstyle=\ttfamily\footnotesize,
    breakatwhitespace=false,         
    breaklines=true,                 
    captionpos=b,                    
    keepspaces=true,                 
    numbers=left,                    
    numbersep=5pt,                  
    showspaces=false,                
    showstringspaces=false,
    showtabs=false,                  
    tabsize=4,
    frame=single, % Adds a frame around the code
    language=Python % SageMath is based on Python
}
\begin{document}
\title[monogeneity and Galois group] {Arithmetic Aspects of Number Fields Generated by Polynomial Families}
\author[Rupam Barman]{Rupam Barman\, \orcidlink{0000-0002-4480-1788}}
\address[Barman]{Department of Mathematics, Indian Institute of Technology Guwahati, Assam, India,
PIN- 781039}
\email{ rupam@iitg.ac.in}
\author[A. Jakhar]{Anuj Jakhar\, \orcidlink{0009-0007-5951-2261}}
\address[Jakhar, Kalwaniya]{Department of Mathematics, Indian Institute of Technology Madras, Chennai-600036, Tamil Nadu, India}
\email{anujjakhar@iitm.ac.in}
\email{ravikalwaniya3@gmail.com}

\author[R.~Kalwaniya]{Ravi Kalwaniya\, \orcidlink{0009-0008-6964-5276}}

\author[P.~Yadav]{Prabhakar Yadav\, \orcidlink{0009-0000-9622-3775}}
\address[Yadav]{Department of Computer Science, Ashoka University, Sonipat-131021, Haryana, India}
\email{pkyadav914@gmail.com; yadavprabhakar096@gmail.com}

\begin{abstract}
Let $f(x)=(x^{k}+c)^{m}-ax^{n}\in\mathbb{Z}[x]$ be an irreducible polynomial over $\mathbb{Q}$, where $k,m,n\in\mathbb{N}$ with $km>n$, and let $K=\mathbb{Q}(\theta)$, where $\theta$ is a root of $f(x)$. We investigate the arithmetic properties of the number fields that arise from this family. We first obtain an explicit formula for the discriminant of $f(x)$. Using this formula, we establish necessary and sufficient conditions for the monogeneity of $f(x)$, expressed in terms of the prime divisors of $a$ and $c$ and the parameters $k,m,n$. This yields infinite families of monogenic polynomials of arbitrary degree, including families with a non-square-free discriminant. Building on these results, we extend our algebraic characterization to composite polynomials, establishing some explicit conditions for the monogeneity of the composition of $f(x)$ with an arbitrary polynomial $g(x)$.   From an analytic point of view, we derive asymptotic estimates for the number of monogenic polynomials in these families under natural assumptions. We further study non-monogeneity via the field index $i(K)$ and, for each prime $p$, provide sufficient conditions ensuring $\nu_p(i(K))=1$, yielding partial progress toward a problem of Narkiewicz. We also highlight a connection with a class of differential equations naturally associated with $f(x)$. As an application, we determine the conditions under which the splitting field of $f(x)$ has a full symmetric Galois group. Several explicit examples illustrate our results.
\end{abstract}
{\iffalse
\begin{abstract}

Let $f(x)=(x^{k}+c)^{m}-ax^{n}\in\mathbb{Z}[x]$ be an irreducible polynomial over the field $\Q$ of rationals,
where $k,m,n\in\mathbb{N}$ with $km>n$, and let $K=\mathbb{Q}(\theta)$ with $\theta$ a root of $f(x)$.
We study the arithmetic of the number fields arising from this family with emphasis on monogeneity, index theory, and Galois structure.
We first compute the discriminant of $f(x)$ explicitly and derive the necessary and sufficient conditions, depending only on the prime divisors of $a,c$ and the exponents $k,m,n$, for $\mathbb{Z}[\theta]$ to coincide with the full ring of integers $\mathbb{Z}_K$.
This yields infinite families of monogenic polynomials of arbitrary degree, including families with a non-square-free discriminant. Building on these results, we extend our algebraic characterization to composite polynomials, establishing explicit conditions for the monogeneity of the composition of $f(x)$ with an arbitrary polynomial $g(x)$. We further determine some conditions under which the associated splitting fields have a Galois group equal to the full symmetric group.
From a quantitative perspective, we obtain an asymptotic bound for the number of monogenic polynomials in these families under certain conditions.
Finally, we investigate non-monogeneity via the field index $i(K)$ and provide, for each rational prime $p$, sufficient conditions ensuring $\nu_p(i(K))=1$, giving partial progress toward a problem of Narkiewicz.
Several explicit examples illustrate our results.
\end{abstract}
\fi}
\maketitle

\tableofcontents

\section{Introduction } \label{intro}

A number field $K$ is said to be \emph{monogenic} if its ring of integers $\mathbb{Z}_K$ is generated as a $\mathbb{Z}$-algebra by a single element, that is, 
\[
\mathbb{Z}_K=\mathbb{Z}[\alpha]\quad \text{for some } \alpha\in\mathbb{Z}_K.
\]
Equivalently, $K=\mathbb{Q}(\alpha)$ admits a power integral basis $\{1,\alpha,\dots,\alpha^{n-1}\}$, where $n=[K:\mathbb{Q}]$.
monogeneity greatly simplifies arithmetic computations in number fields such as the determination of discriminants, integral bases, and ideal arithmetic. Deciding whether a number field is monogenic is a longstanding and challenging problem in algebraic number theory.

Let $f(x)\in\mathbb{Z}[x]$ be a monic irreducible polynomial of degree $n$, let $\theta$ be a root of $f(x)$, and set $K=\mathbb{Q}(\theta)$.
Denote by $\D_f$ and $d_K$ the discriminant of the polynomial $f(x)$ and the number field $K$, respectively.
It is well known that
\begin{equation}\label{eq:disc-relation}
\D_f=[\mathbb{Z}_K:\mathbb{Z}[\theta]]^2\, d_K.
\end{equation}
If $\mathbb{Z}[\theta]=\mathbb{Z}_K$, then $f(x)$ is called \emph{monogenic}, and in this case $\theta$ generates a power integral basis of $K$.
While the monogeneity of $f(x)$ implies the monogeneity of $K$, the converse does not hold in general, as classical examples already demonstrate.

The problem of determining whether a number field is monogenic is classical and dates back to the foundational work of Dedekind and Hasse.
In 1878, Dedekind introduced what is now known as \emph{Dedekind’s  criterion}, which gives a necessary and sufficient condition for a prime $p$ not to divide the index $[\mathbb{Z}_K:\mathbb{Z}[\theta]]$ in terms of the factorization of $f(x)$ modulo $p$ (cf.\ \cite[Theorem~6.1.4]{HC}, \cite{RD1878}).
This criterion has generated substantial interest, and several equivalent formulations and generalizations have since been developed (see, for example, \cite{MEC, JNT, Jh-Kh}). It also led to further work on the arithmetic of indices and on the structure of rings of integers.

From a qualitative perspective, significant progress has been made in classifying monogenic number fields in special families.
Ga\'al developed the theory of index form equations and obtained classifications for number fields of small degree (cf.\ \cite{Gaal}).
More recently, Jakhar and collaborators derived necessary and sufficient conditions for the monogeneity of various families of trinomials and related polynomials by refining Dedekind’s criterion and incorporating local considerations (see \cite{JNT, jKN, J-K-K, J-K-Y}).
These works also produced infinite families of monogenic number fields of arbitrarily large degree.
Further developments and related results can be found in \cite{Jones, JonesHarr, BNW2025,BNW2026, S-L-S, Smith2021, GR2017, KKR}; we refer the reader to the survey article of Ga\'al \cite{Gaal} for a comprehensive overview.

Alongside the study of individual polynomial families, the monogeneity of composite polynomials has recently emerged as a particularly active direction of research. Understanding when a composition $f(g(x))$ remains monogenic is of intrinsic interest: it reveals subtle relationships between the discriminants and indices of the constituent polynomials and provides a systematic mechanism for constructing infinite towers of monogenic number fields from simpler base cases. At the same time, the problem is notoriously delicate, as algebraic invariants behave in intricate ways under composition. Despite these difficulties, significant progress has been made. Jones and Harrington \cite{JonesHarr} investigated composed binomials; Sharma and collaborators \cite{S-L-S}, as well as Smith \cite{hanson}, studied iterated families; and Kaur et al.\ \cite{KKR} established precise criteria for polynomials of the form $f(x^k)$. Motivated by these advances, determining explicit arithmetic conditions under which both $f(x)$ and $f(g(x))$ remain monogenic continues to be a central problem.

Parallel to this qualitative theory, a quantitative perspective on monogeneity has emerged.
A sufficient condition for monogeneity is the squarefreeness of the discriminant $\D_f$, as follows immediately from \eqref{eq:disc-relation}.
However, this condition is far from necessary.
In 2012, Kedlaya \cite{kedlaya} constructed, for each integer $n\ge 2$, infinitely many degree-$n$ monic irreducible polynomials $f(x)\in\mathbb{Z}[x]$ with squarefree discriminant.
At the same time, several authors have exhibited infinite families of monogenic polynomials with non-squarefree discriminant.
In particular, Jones and White \cite{L3} studied the monogeneity of trinomials of the form $x^n+Ax^m+B$ and obtained asymptotic formulas for the number of monogenic trinomials with bounded coefficients.
Jones \cite{L1} further constructed infinitely many monogenic polynomials of prime degree whose discriminants are not squarefree, relying in part on $abc$-conjecture.
More recently, breakthrough results of Bhargava and collaborators \cite{Manjulcubic, Manjulquartic} showed that a positive proportion of number fields are not monogenic even in the absence of any local obstruction.

A closely related invariant is the \emph{index of a number field} $K$, defined by
\[
i(K)=\gcd\{[\mathbb{Z}_K:\mathbb{Z}[\alpha]] \mid \alpha\in\mathbb{Z}_K,\ K=\mathbb{Q}(\alpha)\}.
\]
If $K$ is monogenic, then necessarily $i(K)=1$, although the converse does not hold in general.
The systematic study of index divisors dates back to Dedekind \cite{RD1878}, who characterized the prime divisors of $i(K)$ in terms of the factorization of rational primes in $\mathbb{Z}_K$ and constructed the first example of a cubic field with $i(K)=2$.
Subsequently, Bauer \cite{Bauer1907} showed that for any integer $n$ and any prime $p<n$, there exists a number field $K$ of degree $n$ such that $p\mid i(K)$, while Zylinski \cite{Zylinski1913} proved that if $p\mid i(K)$, then necessarily $p<[K:\mathbb{Q}]$.
These results demonstrate that the index captures genuinely global arithmetic information that is not determined solely by the degree or discriminant.
A deeper understanding of the $p$-adic valuation $\nu_p(i(K))$ was initiated by Ore \cite{Ore1927}, who conjectured that the decomposition type of $p$ in $\mathbb{Z}_K$ does not, in general, uniquely determine $\nu_p(i(K))$.
This conjecture was partially resolved by Engstrom \cite{Engstrom1930}, who showed that for number fields of degree $n\le 7$, the prime ideal factorization of $p$ completely determines $\nu_p(i(K))$ and explicitly computed these valuations.
Later refinements and extensions were obtained by \'Sliwa \cite{Sliwa1982} for degrees up to $12$ under suitable hypotheses, and by Nart \cite{Nart1985}, who provided a $p$-adic characterization of the index.
Despite further progress in special cases (see, for example, \cite{GPP1991, Nak1983, PP2012, SW2003}), a general explicit description of $\nu_p(i(K))$ remains unknown.
This difficulty is formalized in Problem~22 of Narkiewicz \cite{Narkiewicz2004}, which asks for an explicit description of the highest power of a prime $p$ dividing $i(K)$.

Another fundamental aspect of the arithmetic of number fields is the structure of their Galois groups.
Hilbert’s irreducibility theorem guarantees the abundance of irreducible polynomials with Galois group $S_n$ or $A_n$. But imposing additional arithmetic constraints such as monogeneity or prescribed index behavior makes the problem substantially more delicate.
Only a limited number of constructions are known that simultaneously control monogeneity, discriminant structure and Galois group.

Beyond their intrinsic algebraic importance, monogenic polynomials also provide striking structural applications in analysis. When the auxiliary polynomial associated with a higher-order linear differential equation is monogenic, the existence of a power integral basis allows the characteristic roots to be expressed as integer linear combinations of a single algebraic integer. This arithmetic rigidity translates directly into analytic structure, yielding explicit descriptions of the general solution in terms of exponential functions with integral parameters.

\medskip

Motivated by these themes, in this paper, we study these problems for the family of irreducible polynomials of the type
\[
f(x)=(x^k+c)^m-ax^n\in\mathbb{Z}[x],
\]
where $k,m,n\in\mathbb{N}$ with $km > n$.
This family generalizes classical binomial and trinomial polynomials, which arise as special cases (for instance, when $a=0$ and $m=1$, or when $m=1$), and exhibits diverse arithmetic and Galois-theoretic behavior.

Our main contributions are summarized as follows.
We first compute the discriminant $\D_f$ of $f(x)$ explicitly and obtain necessary and sufficient conditions for monogeneity involving only the prime divisors of $a$, $c$, and the exponents $k,m,n$. 
 This yields a precise arithmetic criterion for when the number field generated by a root of $f(x)$ admits a power integral basis.
We then extend this characterization to polynomial compositions, establishing explicit conditions under which the composition $f(g(x))$ remains monogenic for an arbitrary polynomial $g(x)$. This provides a systematic method for constructing new monogenic families.
Next, under suitable hypotheses, we further derive quantitative estimates for the number of monogenic polynomials in these families. We also produce infinite families of monogenic polynomials within this class, including examples with non-squarefree discriminant, and determine conditions ensuring that the associated Galois group is the full symmetric group. 
From the complementary perspective of non-monogeneity, we analyze the field index $i(K)$ and, for each rational prime $p$, give explicit sufficient conditions guaranteeing that $\nu_p(i(K))=1$, thereby contributing partial progress toward Problem~22 of Narkiewicz \cite{Narkiewicz2004}.
Finally, as an application of our main results, we describe the structure of the solutions to a class of linear differential equations that are intimately related to polynomials of the form $f(x)$.
Several explicit examples are included to illustrate our results.

\section{Main Results}
Throughout the paper, $\D_f$ will stand for the discriminant  of  $f(x)=(x^k+c)^m-ax^n$, $km > n \geq 1$,~$t$ for $\gcd(n,k)$ and $n_1, k_1$ for $\frac{n}{t}, \frac{k}{t},$ respectively. More precisely, we prove the following results. 
\begin{theorem} \label{disc of f(x)}
Let $a,c \in \mathbb{Z}$ with $c \neq 0$, and let $m,n,k \in \mathbb{N}$ with $km>n$. 
Suppose that the polynomial $f(x) = (x^k + c)^m - a x^n \in \mathbb{Z}[x]$ is irreducible over $\mathbb{Z}$ and let $\theta$ be a root of $f(x)$. Then %This form is accurate other form of writing will be more messy.
\begin{equation}
    \D_f
=(-1)^{\binom{km}{2}+(km+n+t)(m-1)}\,a^{\,k(m-1)}c^{m(n-1)}\,
\left[\,(km)^{\,k_1m}c^{k_1m-n_1} - a^{k_1}n^{\,n_1}\,(km-n)^{\,k_1m-n_1}\right]^{t}.
\end{equation}
where $n = n_1 t$, $k = k_1 t$ and $t=\gcd(n,k)$.
\end{theorem}
We shall denote by $\calC,\, \calE$, the integers defined by
\begin{equation}\label{CE}
	\calC=(km)^{\,k_1m}c^{k_1m-n_1}  ~{\rm and}  ~\calE= a^{k_1}n^{\,n_1}\,(km-n)^{\,k_1m-n_1},
\end{equation}
then in view of the above theorem $\D_f = (-1)^{\binom{km}{2}+(km+n+t)(m-1)}\,a^{\,k(m-1)}c^{m(n-1)}\,(\calC-\calE)^{t}.$

\begin{theorem} \label{monogeneity of f(x)}
	Let $K=\mathbb Q(\theta)$ be an algebraic number field with $\theta$ in the ring $\Z_K$ of algebraic integers of $K$ having  minimal polynomial   $f(x) = (x^k+c)^m-ax^n$ of degree $km$ over $\mathbb Q$ with   $km>n\geq{1}$. Let $\gcd(n,k) = t$ and $n = n_1t,\, k = k_1t.$ A prime divisor $p$ of the discriminant  $\D_f$ of $f(x)$ does not divide $[\Z_{K} : \Z[\theta]]$ if and only if $p$  satisfies one of the following conditions: 
	 \begin{enumerate}[label=\textup{(\roman{enumi})}]
    \item  If $p\mid c$  and $p\mid a$, then $m=1$ and $p^2\nmid c$;
    \item  If $p\mid c,~~p\nmid a$ and $n=1$  with $j\geq 1$ as the highest power of $p$ dividing $km-1$, then either $p\mid c_2$ and $p\nmid a_1$ or $p\nmid c_2[a^ka_1^{\,km-1}-(-mac_2)^{\,km-1}]$, where $c_2=\frac{c}{p}$ and $a_1=\frac{a^{p^j}-a}{p}$;
    \item  If $p\mid c,~~p\nmid a$ and $n\geq 2$ with $\ell\geq 0$ as the highest power of $p$ dividing $k-n$, then $m=1$ and  either
$p\mid a_{1}$ and $p\nmid c_{2}$, or $p\nmid a_{1}c_{2}^{\,n-1}
\bigl[a^{n_1}a_{1}^{\,k_1-n_1}-(-c_{2})^{\,k_1-n_1}\bigr]$, where $c_2=\frac{c}{p}$ and $a_1=\frac{a^{\,p^{\ell}}-a}{p}$;
    \item  If $p\nmid c$ and $p\mid a$ and further;
    \begin{itemize}
        \item[\textup{(a)}] If $m\geq 2$, then $p^2\nmid a$;
        \item [\textup{(b)}] If $m=1$ with $\ell\geq 1$ as the highest power of $p$ dividing $k$, then either $p\mid a_2$ and $p\nmid c_1$ or $p\nmid a_2\bigl[a_{2}^{k_{1}}(-c)^{n_{1}}-(-c_{1})^{k_{1}}\bigr]$, where $a_2=-\frac{a}{p}$ and $c_1=\frac{(-c)^{\,p^{\ell}}+c}{p}$;
    \end{itemize}
    \item If $p \nmid ca$ and $p \mid k$, where $k = p^{j}s'$ and $n = p^{j}s$ with 
$p \nmid \gcd(s,s')$, then the polynomials $\frac{1}{p}[a_1\,x^{n} - p m\, t(x)(x^{k} + c)^{\,m-1}]$ and $(x^{s'} + c)^m - a x^s$ are coprime, where $a_1=\frac{a^{\,p^{j}}-a}{p}$ and\\ $t(x) = \frac{1}{p} \left( \sum_{i=0}^{p^{j}-1} \binom{p^{j}}{i} x^{\,s' \,i} \,c^{\,p^{j}-i} - c \right)$;
     \item If $p\nmid cak$ and $p\mid m$, then $p \mid n$ and $p^2\nmid (a^{p}-a)$;
    \item  If $p\nmid cakm$, then $p^2\nmid (\calC-\calE).$
	\end{enumerate}
\end{theorem}    
We want to point out here that the above theorem for $k=2,~c=1$ and $m=n$ yields the main result of \cite{BNW026}.\\
The following corollaries are an immediate consequence of  Theorem $\ref{monogeneity of f(x)}$. 
\begin{corollary}\label {Cor:1.3}
	Let $K = \Q(\theta)$ and $f(x) = (x^k+c)^m-ax^n$ be as in Theorem \ref{monogeneity of f(x)}. Then $\Z_{K} = \Z[\theta]$ if and only if each prime $p$ dividing $\D_f$ satisfies one of the conditions \textup{(i)$-$(vii)} of Theorem \ref{monogeneity of f(x)}.
\end{corollary}
%For any nonzero integer $n$, we write $\operatorname{rad}(n)$ for the square-free part of $n$, that is, the product of the distinct prime divisors of $n$.

\begin{definition}Let $q$ be a prime number. A polynomial $a_nx^n+a_{n-1}x^{n-1}+\cdots+a_{0}\in \Z[x]$ with $a_n\neq 0$ is called a $q-$Eisenstein polynomial  if $q\nmid a_n,~q\mid a_i$ for $0\leq i\leq n-1$ and $q^2\nmid a_{0}$.
\end{definition}
The following corollary provides an infinite family of monogenic  polynomials. 
The proof is given in Section~\ref{sec:4}. This plays a crucial role in deriving analytic results on the monogeneity of $f(x)$.
\begin{corollary}\label{6.1lemma}
Let $q$ be a prime such that $q \mid\mid a$ and let $m \ge 2$. Then $f(x) = (x^{q^j} \pm 1)^m - a x^n$ is irreducible. Furthermore, a prime divisor $p$ of the discriminant  $\D_f$ of $f(x)$ does not divide $[\Z_{K} : \Z[\theta]]$ if and only if $p$  satisfies one of the following conditions: 
\begin{enumerate}
    \item If $p\mid a$, then $p^2\nmid a$;
    \item If $p \nmid a$ and $p \mid m$, then $p\mid n$ and  $p^2 \nmid (a^p - a)$;
    \item If $p \nmid am$, then $p^2 \nmid (\mathcal{C} - \mathcal{E})$.
\end{enumerate}
\end{corollary}
 We now define a set $\scrE_f$ for the polynomial $f$ which will be used in the proof of Theorem~\ref{monogeneity of f(g(x))}.

\begin{definition}
Let $f(x) =(x^k+c)^m-ax^n \in \mathbb{Z}[x]$. 
We define the \emph{set of exceptional primes} \(\scrE_f\) associated with \(f(x)\) by
\[  \scrE_f = \left\{\, p \text{ prime} : p \mid \D_f \text{ and } p \nmid cak m \right\}\cup\left\{\, p \text{ prime} : p \mid \D_f \text{ and } m=1,\,p\mid c \textit{ and } p\nmid a(k-n)\right\}, \]
where \(\D_f\) denotes the discriminant of \(f(x)\).
\end{definition}
The next theorem provides conditions under which $f(g(x))$ is monogenic, where $g(x)$ is an arbitrary polynomial.

 \begin{theorem}\label{monogeneity of f(g(x))}
    Let $f(x) = (x^k+c)^m-ax^n\in \mathbb{Z}[x]$ be irreducible over $\mathbb{Q}$ and let  $g(x)\in \Z[x]$ be a polynomial of degree $d$ such that \(T(x) := f(g(x))\) is irreducible over $\mathbb{Q}$.  Let $K = \mathbb{Q}(\theta)$, where $\theta$ is a root of $T(x)$. Assume that $\rad(k)\mid ca$ and that, for every prime $p \in \scrE_f$, 
we have $p \mid g(0)$, $p \nmid d$, and the polynomial $g(x)$ satisfies \( g'(x) \equiv d x^{d-1} \pmod{p}.\) Furthermore, in the case $n = 1$, we assume that for every prime $p$ satisfying 
$p \mid \D_f$, $p \mid c$, and $p \nmid a$, one has $p \nmid \D_g$.
 Then, the following statements are equivalent:
    %when $n=1$ and $p$ is prime such that $p\mid \D_f,$ $p \mid c$ and $p \nmid a$, then $p \nmid \D_g$. 
   \begin{itemize}
    \item[1.] No prime dividing $\D_f$ divides $[\Z_K:\Z[\theta]]$,
    \item[2.] $f(x)$ is monogenic.
\end{itemize}
    \end{theorem}
The following theorem provides an asymptotic formula for the number of such monogenic polynomials as defined in Theorem~\ref{monogeneity of f(x)}.

\begin{theorem} \label{1.1:thm}
        Let $n \geq 2,m \geq 2$ be positive integers such that there exists a prime $q \mid n$ with $q \nmid m$ and $q^{v_q(n)}m > n$. Let $j = v_q(n)$ and  $\sfrak=\frac{n}{q^j}$. Then there are
        \begin{align} \label{1.1:eqn}
            \frac{X}{\gamma^2 \zeta(2)} \prod_{p \mid qm} \left( 1- \frac{1}{p^2}\right)^{-1} \prod_{p \nmid qm\sfrak (m-\sfrak)} \left( 1-\frac{1}{p^2-1} \right) + O(X^{3/4}).
        \end{align}
        monogenic polynomials $f(x)= (x^{q^j}\pm 1)^{m} - ax^n$ satisfying $a \leq X$. The $O$-constant in \eqref{1.1:eqn} may depend on $n$ and $m$.
\end{theorem}
 Our next result proves that there are infinitely many monogenic polynomials $f(x)$. Part of the proof relies on the $abc$-conjecture for the number field. 
\begin{theorem} \label{2.10:thm}
    Let $n,m$ and $k$ be integers such that $\gcd(km,n)=1$ and $\rad(k)=q$, where $q$ is a fixed prime. Assume $abc$-conjecture for number fields, then there exist infinitely many primes $p$ such that $f(x) = (x^k \pm 1)^m - p x^n$ is monogenic. 
\end{theorem}
The next theorem establishes the nonmonogeneity of $f(x)$ and determines its index explicitly in certain special cases.

\begin{theorem}\label{nonmonogeneity}
Let $p$ be an odd prime and let $m >1 ,k \ge 1$ be integers such that $p-1$ divides $(km-k-1)$ and $p \nmid (km-k-1)$. 
Let $a,c \in \mathbb{Z}$ with $a \equiv 1 \pmod{p}$ and $p \mid c$, and let $f(x)=(x^k+c)^m-ax^{k+1}$
be irreducible over  $\mathbb{Q}$ with $mk>k+1$. Let $K=\mathbb{Q}(\theta)$, where $\theta$ is a root of $f(x)$. 
If \[v_p(c)>(k+1)v_p(m) \textit{ and } \gcd\!\left(v_p(c)-v_p(m),\,k\right)=1,\] then $v_p\!\left(i(K)\right)=1.$
\end{theorem}
As an application of Theorem~\ref{monogeneity of f(x)}, we obtain the following result regarding the solvability of a certain class of linear differential equations.
    \begin{theorem}{\label{t2}}
    Let \begin{equation}{\label{diff}}
        \left(\frac{d^k }{dx^k}+c\right)^m\!y-a\frac{d^ny }{dx^n}=0
    \end{equation}
    be a differential equation, where $km>n\geq 1$. Let $\mathcal{F}(z) =  (z^k+c)^m-az^n$ be the irreducible auxiliary equation of \eqref{diff}
  with a root $\theta$. Suppose $F$ is such that $K=\Q(\theta)$ is the splitting field of $F$. If each prime $p$ dividing the discriminant $\D_{\mathcal{F}}$
 of $\mathcal{F}(z)$ satisfies any one of the conditions \textup{(i)} to \textup{(vii)} of Theorem \ref{monogeneity of f(x)}, then the general
 solution of the given differential equation \eqref{diff} is of the form
\begin{align*}
    y(x) = \sum_{i=1}^{km} \alpha_i\prod_{j=1}^{km}e^{c_{j-1}^{(i)} \theta^{j-1}x},
\end{align*}
where $c_{j-1}^{(i)}$ are integers and $\alpha_i$ are arbitrary real constants for all $1\leq i,j \leq km$.
\end{theorem}
The following theorem gives a criterion under which the Galois group of $f(x)$ over $\mathbb{Q}$ is the full symmetric group $S_n$.

\begin{theorem} \label{galois group}
Let $f(x) = (x^k+c)^m - a x^n \in \Z[x]$ be an irreducible polynomial of degree $d = km$. Assume that the following conditions hold:
\begin{enumerate}
    \item The integer $q = km-n$ is a prime number that satisfies $\frac{d}{2} < q < d-2$.
    \item There exists a prime $p$ such that $p \mid c$, $p \mid a$, and $p^2 \nmid a$.
    \item The exponents satisfy the fact that $k(m-1)$ is an odd integer.
    \item There exists a prime $\ell$ such that $v_\ell(a)$ is odd and $\ell \nmid kmc$.
\end{enumerate}
Then the Galois group of $f(x)$ over $\Q$ is the symmetric group $S_{d}$.
\end{theorem}
\section{Discriminant of \texorpdfstring{$f(x)$}{f(x)}}
Now we state two lemmas that will be useful in the proof of Theorem~\ref{disc of f(x)}. 
Let $\theta$ be a root of an irreducible monic polynomial $f(x) \in \mathbb{Z}[x]$. 
The discriminant of $f(x)$ and the norm of $f'(\theta)$ are related by the following formula (see~\cite[Lemma~2.6]{SKM}).

\begin{lemma}\label{2.3} Let $f(x)\in\Z[x]$ be monic and irreducible polynomial with $\deg(f)=n.$ Let $f(\theta)=0$ and $K=\Q(\theta).$ Then
	$$ D=(-1)^{\binom{n}{2}}\,\mathcal{N}_{K/\Q} (f'(\theta)).$$
\end{lemma}
Let $K$ be a finite extension of a field $F$. For any element $\alpha \in K$, computing the norm of $\alpha$ with respect to $K/F$ is closely related to computing the norm of $\alpha$ with respect to $F(\alpha)/F$, as given by the following formula (see~\cite[Theorem~1.20]{SKM}).

\begin{lemma}\label {2.4}Let $K/F$ be an extension of degree $n$ and $\alpha$  be an element of $K$ with $[F(\alpha):F]=d$. Then $$ \mathcal{N}_{K/F}(\alpha)=\left( \mathcal{N}_{F(\alpha)/F}(\alpha)\right)^{n/d}. $$
\end{lemma}

\begin{proof}[Proof of Theorem~\ref{disc of f(x)}]
Let $\theta$ be a root of $f(x)=(x^k+c)^m-ax^n$ and set $K=\mathbb{Q}(\theta)$. 
Since $f(\theta)=0$, we have
\begin{equation}\label{eq:1}
(\theta^k + c)^m = a\,\theta^n.
\end{equation}

Differentiating $f(x)$ gives
\[
f'(x)=km\,x^{k-1}(x^k+c)^{m-1}-an\,x^{n-1}.
\]
Multiplying by $\theta$ and using \eqref{eq:1}, we obtain
\begin{align}
\theta f'(\theta)
&=km\,\theta^k(\theta^k+c)^{m-1}-an\,\theta^n \nonumber\\
&=(\theta^k+c)^{m-1}\bigl(km\theta^k-n(\theta^k+c)\bigr)\nonumber\\
&=(\theta^k+c)^{m-1}\bigl((km-n)\theta^k-nc\bigr).
\label{eq:2}
\end{align}

Let $\mathcal{N}$ denote the norm $\mathcal{N}_{K/\mathbb{Q}}$. 
Since the constant term of $f(x)$ is $c^m$, it follows that, $\mathcal{N}(\theta)=(-1)^{km}c^m.$ Taking norms in \eqref{eq:2}, we obtain
\begin{equation}\label{eq:3}
(-1)^{km}c^m\,\mathcal{N}(f'(\theta))
=
\mathcal{N}(\theta^k+c)^{m-1}
\mathcal{N}\!\bigl((km-n)\theta^k-nc\bigr).
\end{equation}
We first compute $\mathcal{N}\!\bigl((km-n)\theta^k - nc\bigr)$. Set $z=(km-n)\theta^k-nc$ and $\theta^k=\frac{z+nc}{km-n}.$ Write $t=\gcd(n,k)$ and express $n=n_1t$, $k=k_1t$. Raising \eqref{eq:1} to the $k_1$-th power gives
\[
(\theta^k+c)^{k_1m}=a^{k_1}(\theta^k)^{n_1}.
\]
Substituting $\theta^k=\dfrac{z+nc}{km-n}$ and simplifying yields
\[
(z+kmc)^{k_1m}
=
a^{k_1}(z+nc)^{n_1}(km-n)^{k_1m-n_1}.
\]
Thus $z$ is a root of
\begin{equation}\label{eq:h}
h(x)
=
(x+kmc)^{k_1m}
-
a^{k_1}(x+nc)^{n_1}(km-n)^{k_1m-n_1}
\in\mathbb{Z}[x].
\end{equation}

We claim that $h(x)$ is the minimal polynomial of $z$. 
First observe that $\mathbb{Q}(z)=\mathbb{Q}(\theta^k)$. 
Since $k=k_1t$, we have $\theta^k=(\theta^t)^{k_1}$, so 
$\mathbb{Q}(\theta^k)\subseteq\mathbb{Q}(\theta^t)$. 
Conversely, from \eqref{eq:1} we see that $\theta^n\in\mathbb{Q}(\theta^k)$, 
and since $\gcd(n,k)=t$, there exist integers $v_1,v_2$ with 
$nv_1+kv_2=t$. Hence $\theta^t=(\theta^n)^{v_1}(\theta^k)^{v_2}\in\mathbb{Q}(\theta^k),$ so $\mathbb{Q}(\theta^t)=\mathbb{Q}(\theta^k)$. Now consider $g(x)=(x^{k_1}+c)^m-ax^{n_1}\in\mathbb{Z}[x].$ Then
\[
g(\theta^t)=((\theta^t)^{k_1}+c)^m-a(\theta^t)^{n_1}
=f(\theta)=0.
\]
If $g(x)$ were reducible over $\mathbb{Z}$, say 
$g(x)=g_1(x)g_2(x)$, then
$f(x)=g(x^t)=g_1(x^t)g_2(x^t)$ would also be reducible, contradicting the irreducibility of $f(x)$. Thus $g(x)$ is irreducible and $[\mathbb{Q}(\theta^t):\mathbb{Q}]=k_1m.$ Therefore $h(x)$ is the minimal polynomial of $z$, and
\[
\mathcal{N}_{\mathbb{Q}(z)/\mathbb{Q}}(z)
=
(-1)^{k_1m}h(0),
\]
where
\[
h(0)
=
(kmc)^{k_1m}
-
a^{k_1}(nc)^{n_1}(km-n)^{k_1m-n_1}.
\]
Since $[\mathbb{Q}(\theta):\mathbb{Q}]=km$ and 
$[\mathbb{Q}(z):\mathbb{Q}]=k_1m$, we obtain
\begin{equation}\label{eq:3.6}
\mathcal{N}\bigl((km-n)\theta^k-nc\bigr)
=
\bigl((-1)^{k_1m}h(0)\bigr)^t.
\end{equation}
It remains to compute $\mathcal{N}(\theta^k+c)$. Set $\gamma=\theta^k+c$. Arguing as above, the minimal polynomial of $\gamma$ over $\mathbb{Q}$ is
\[
h_2(x)=x^{k_1m}-a^{k_1}(x-c)^{n_1}.
\]
Hence
\begin{equation}\label{eq:3.7}
\mathcal{N}(\theta^k+c)
=
(-1)^{km+n+t}a^{\,k}c^{\,n}.
\end{equation}

Finally, the result follows from Lemma~\ref{2.3}, together with 
\eqref{eq:3}, \eqref{eq:3.6}, and \eqref{eq:3.7}.
\end{proof}

\section{Monogeneity of \texorpdfstring{$f(x)$}{f(x)}}\label{sec:4}
The following lemma is an easy consequence of the Binomial theorem and will be used in the proof of Theorem~\ref{monogeneity of f(x)}.
\begin{lemma}\label{binomial expansion} { Let $k\geq 1$ be the highest power of a prime $p$ dividing a number $n=p^{k}s'$ and $c$ be an integer not divisible by $p$. If $\bar{g}_1(x)\cdots \bar{g}_r(x)$  is the factorization of $x^{s'}-\bar{c}$ into  a product of distinct irreducible polynomials over $\Z/p\Z$ with $g_i(x)\in\mathbb Z[x] $ monic, then 
		$$x^n-c=(g_1(x)\cdots g_r(x)+pH_1(x))^{p^k}+pg_1(x)\cdots g_r(x)H_2(x)+p^2H_3(x)+c^{p^k}-c$$ 
		\noindent for some polynomials  $H_1(x), H_2(x), H_3(x)\in \mathbb Z[x]$.}
\end{lemma}
The following well known theorem will be used in the sequel. The equivalence of
assertions (i) and (ii) of the theorem was proved by Dedekind \cite[Theorem 6.1.4]{CH1993} and the equivalence of (i) and (iii) can be easily deduced by applying Uchida’s result \cite{Uch77}. 
\begin{lemma}\label{dedekind}
	Let  $f(x) \in \Z[x]$ be a monic irreducible polynomial  having factorization $\bar{g}_1(x)^{e_1} \cdots \bar{g}_{t}(x)^{e_{t}}$ modulo a prime $p$ as a product of powers of distinct irreducible polynomials over $\Z/p\Z$ with $g_i(x) \in \Z[x]$ monic. Let $K=\Q(\theta)$ with $\theta$ a root of $f(x)$.   Then the following statements are equivalent:
	\begin{itemize}
		\item[\textup{(i)}] $p$ does not divide $[\Z_K:\Z[\theta]]$.
		\item[\textup{(ii)}] For each $i$, we have $e_i =1 $ or $\overline g_i (x)$ does not divide $\overline M(x)$, where $M(x) = \frac{1}{p}(f(x) -  g_1 (x)^{e_1} \cdots  g_{t} (x)^{e_{t}} )$.
		\item[\textup{(iii)}] $f(x)$ does not belong to the ideal $\langle p, g_i(x)\rangle^2$ in $\Z[x]$ for any $i$, $1\leq i\leq t$.
	\end{itemize}
\end{lemma} 
The following lemma provides a useful decomposition of $f(x)$ in terms of the auxiliary polynomial $h(x)$ defined below; this decomposition will be used in the proof of Theorem~\ref{monogeneity of f(x)}.

\begin{lemma}\label{lemma:3.6}
	Let $f(x)=(x^k+c)^m-ax^n\in\Z[x]$ be a polynomial of degree $km$. Let $p$ be a prime number such that $n=p^js,~k=p^js',$ where $p\nmid\gcd(s,s')$. Define $h(x)=(x^{\,s'}+c)^m-ax^s\in\Z[x]$. Then $f(x)$ can be expressed as \[f(x) =  h(x)^{\,p^{j}}+(a^{\,p^{j}} + a)\,x^{n} - p m\, t(x)(x^{k} + c)^{\,m-1}+ p\,h(x)\,h_{1}(x) + p^{2}\,h_{4}(x),\] for some $h_1(x),\, h_4(x)\in\Z[x]$ and $t(x) = \frac{1}{p} \left( \sum_{i=1}^{p^{j}} \binom{p^{j}}{i} x^{\,s' p^{j}} c^{\,p^{j}-i} - c \right) \in \mathbb{Z}[x].$
\end{lemma}
\begin{proof}
Since  $h(x) = (x^{s'} + c)^m - a x^s.$ We first expand $(x^{s'} + c)^{mp^j}$ using the binomial theorem:
\[
(x^{s'} + c)^{mp^j}
= \bigl( h(x) + ax^s \bigr)^{p^{j}}
= h(x)^{p^{j}} + p\,h(x)h_{1}(x)  + p^{2}h_{2}(x)+ a^{p^{j}}x^{n},
\]
for some $h_{1}(x), h_{2}(x) \in \mathbb{Z}[x]$. On the other hand,
\[
(x^{s'} + c)^{mp^j}
= \bigl( x^{k} + c + p\,t(x) \bigr)^{m},
\]
where
\[
t(x) = \frac{1}{p} \left( \sum_{i=0}^{p^{j}-1} \binom{p^{j}}{i} x^{\,s' i}\, c^{\,p^{j}-i} - c \right) \in \mathbb{Z}[x].
\]

Expanding the last expression again using the binomial theorem gives
\[
\bigl( x^{k} + c + p\,t(x) \bigr)^{m}
= (x^{k} + c)^{m} + p m\, t(x) (x^{k} + c)^{m-1} + p^{2} h_{3}(x),
\]
for some $h_{3}(x) \in \mathbb{Z}[x]$. Comparing the two expansions yields
\[
h(x)^{p^{j}} + p\,h(x)h_{1}(x) + p^{2}h_{2}(x)+ a^{p^{j}}x^{n}
= (x^{k} + c)^{m} + p m\, t(x) (x^{k} + c)^{m-1} + p^{2} h_{3}(x).
\]

Rearranging the terms and using $h(x) = (x^{s'} + c)^m - a x^s$ finally gives
\[
f(x) =  h(x)^{\,p^{j}}+(a^{\,p^{j}} - a)\,x^{n} - p m\, t(x)(x^{k} + c)^{\,m-1}+ p\,h(x)\,h_{1}(x) + p^{2}\,h_{4}(x),
\]
with $h_{4}(x) \in \mathbb{Z}[x]$.
\end{proof}
The following proposition establishes the equivalence between the prime $p$ dividing $\calC-\calE$ and the existence of repeated roots of $f(x)$ modulo $p$, which will be used in the proof of Theorem~\ref{monogeneity of f(x)}(vii).
\begin{proposition}\label{2.06}
	Let $f(x), ~n_1, ~k_1, ~\calC$ and $\calE$ be as defined in Theorem \ref{monogeneity of f(x)}. Let $u_1, u_2\in \Z$ satisfy $ ku_1+nu_2=t$ and $p$ be a prime such that $p\nmid ackm(km-n)$. Furthermore, suppose that there exist integers $ \alpha_1$ and $\alpha_2$ such that $\alpha_1 \equiv \frac{nc}{km-n} \pmod{p^2}$ and $\alpha_2\equiv a^{-1} \left(\frac{kmc}{km-n}\right)^m \pmod{p^2}$. Then the following are equivalent for $i=1,2$.
	\begin{itemize}
		\item[\textup{(i)}] $p^i\mid (\calC-\calE)$;
		\item [\textup{(ii)}] $\alpha_2^{k_1}\equiv\alpha_1^{n_1}\pmod{p^i}$;
		\item [\textup{(iii)}] If $\beta$ satisfies $\beta^t \equiv \alpha_1^{u_1} \alpha_2^{u_2} \pmod{p^i}$, then $\beta$ is a repeated root of $f(x)$ modulo $p^i$.
	\end{itemize}
\end{proposition}
\begin{proof}
	We provide the proof of the above equivalence for $i=1.$ A similar argument will prove the case for $i=2.$ Using $p \nmid ackm(km-n)$, it is easy to note that $\alpha_1 \alpha_2 \not \equiv 0 \pmod{p}$\\
		\noindent (i)$\Leftrightarrow $ (ii) Since $p\mid (\calC-\calE)$, we have
	\begin{align*}
		(km)^{k_1m} c^{k_1m-n_1} \equiv a^{k_1} n^{n_1} (km-n)^{k_1m-n_1} \pmod{p}.
	\end{align*}
	Since $p \nmid ac(km-n)$, the above equation can be rewritten as
	\begin{equation}
		\left( \frac{nc}{km-n} \right)^{n_1} \equiv \left( a^{-1} \left( \frac{kmc}{km-n} \right)^m \right)^{k_1} \pmod{p}.
	\end{equation}
	Substituting the values of $\alpha_1$ and $\alpha_2$ in the above equation, we get
	\begin{align*}
		\alpha_1^{n_1}\equiv\alpha_2^{k_1}\pmod{p}.
	\end{align*}
	Note that all the above arguments are if and only if, so the converse is also true.\\\\
		\noindent  (ii) $\Rightarrow $ (iii)
        We show that $\beta$ with $\beta^t\equiv\alpha_1^{u_1}\alpha_2^{u_2}\pmod {p}$ is a repeated root of $f(x)$ modulo $p.$ It is easy to observe that here $\beta \not\equiv 0 \pmod{p}$.
		First, we show that $f(\beta)\equiv 0\pmod{p}.$
		\begin{align*}
			f(\beta)&=(\beta^k+c)^{m}-a\beta^n,\\
			&=\left((\beta^t)^{k_1}+c\right)^{m} -a(\beta^t)^{n_1}.
		\end{align*}
		Using $\beta^t\equiv \alpha_1^{u_1}\alpha_2^{u_2}\pmod {p}$ in the above equation, we have
		\begin{equation}\label{.002}
            f(\beta)\equiv \left((\alpha_1^{u_1}\alpha_2^{u_2})^{k_1}+c\right)^{m} -a(\alpha_1^{u_1}\alpha_2^{u_2})^{n_1} \pmod{p}.
		\end{equation}
		Using $k_1u_1 + n_1u_2 = 1$ and $\alpha_2^{k_1}\equiv\alpha_1^{n_1}\pmod{p}$ in \eqref{.002}, we get
		\begin{equation*}
			f(\beta) \equiv (\alpha_1+c)^{m} -a \alpha_2\pmod{p}.
		\end{equation*}
		Substituting the values of $\alpha_1$ and $\alpha_2$ in the above equation, we see that
		\begin{align*}
			f(\beta)&\equiv \left(\frac{kmc}{km-n}\right)^m -\left(\frac{kmc}{km-n}\right)^m\pmod{p},\\
			&\equiv0\pmod{p}.
		\end{align*}
		Now, we show that $f'(\beta)\equiv 0 \pmod{p}$. Taking the derivative of $f(x)$ and substituting $x=\beta$, we obtain
		\begin{align*}
				f'(\beta)&\equiv km\beta^{k-1}(\beta^{k}+c)^{m-1} - an\beta^{n-1}\pmod{p},\\
					&\equiv \beta^{-1} \left(km\beta^{k_1t}(\beta^{k_1t}+c)^{m-1} - an\beta^{n_1t} \right)\pmod{p},\\
				&\equiv \beta^{-1} \left(km(\alpha_1^{u_1}\alpha_2^{u_2})^{k_1}((\alpha_1^{u_1}\alpha_2^{u_2})^{k_1}+c)^{m-1} -an(\alpha_1^{u_1}\alpha_2^{u_2})^{n_1}\right)\pmod{p}.
				\end{align*}
				Using $n_1u_2+k_1u_1=1$ and  $\alpha_{1}^{n_1}\equiv\alpha_2^{k_1}\pmod{p}$ in the above equation, we have
				\begin{align*}
				f'(\beta)	\equiv \beta^{-1} \left(km\alpha_{1}(\alpha_{1}+c)^{m-1} - n a\alpha_2 \right)\pmod{p}.
				\end{align*}		
				Substituting values of $\alpha_1$ and $\alpha_{2}$ in the above equation, we obtain
				\begin{align}
				f'(\beta)&\equiv \beta^{-1} \left( \frac{kmnc}{km-n} \left( \frac{nc}{km-n} + c \right)^{m-1} -n \left( \frac{kmc}{km-n} \right)^m \right) \pmod{p},\nonumber\\
				&\equiv \beta^{-1} \left(n\left(\frac{kmc}{km-n}\right)^m - n \left(\frac{kmc}{km-n}\right)^{m}\right) \pmod{p},\nonumber\\
				&\equiv 0 \pmod{p}.
				\end{align}
		\noindent  (iii) $\Rightarrow $ (ii)
		Suppose that $\beta$ is any repeated root of $\bar f (x)= ( x^k+\bar c)^m - \bar a x^n$ in algebraic closure of $\Z/p\Z.$  Note that $\beta \neq \bar{0}$ is as $p\nmid ac$. Then
		\begin{equation}\label {0.008}
			\overline{f}(\beta) = ({\beta}^k+\bar c)^m -\bar a {\beta}^n = \bar{0};~~	\overline{f'}(\beta)= \bar k \overline m \beta^{k-1}({\beta}^k+\bar c)^{m-1} -\bar a \bar{n}{\beta}^{n-1} = \bar{0}. 
		\end{equation}
			On substituting $ \bar a {\beta}^{n} = ({\beta}^k+\bar c)^m$ into \eqref{0.008}, we see that
		\begin{equation}
			\frac{1}{\beta}(\beta^k+c)^{m-1} \left[ (km-n) \beta^k - nc\right] \equiv 0 \pmod{p}.
		\end{equation}
		Observe that $(\beta^k+c) \not\equiv 0\pmod{p}$, otherwise in view  of \eqref{0.008}, we obtain $\beta=\bar 0$, which is not possible. Therefore, keeping in mind that $p\nmid (km-n)$, we have
		\begin{align}\label{0.009}
			\beta^k &\equiv \frac{nc}{km-n}\pmod{p},\nonumber\\
			&\equiv \alpha_1 \pmod{p}.
		\end{align}
Changing $k$ by $tk_1$ and substituting the value $\beta^t \equiv \alpha_1^{u_1}\alpha_2^{u_2}\pmod {p}$ into \eqref{0.009}, we have
		\begin{align*}
			(\alpha_1^{u_1}\alpha_2^{u_2})^{k_1} \equiv \alpha_1 \pmod{p}.
		\end{align*}
		Taking the power $n_1$ on both sides of the  above equation, we obtain
		\begin{equation}\label{.a}
			(\alpha_1^{u_1}\alpha_2^{u_2})^{k_1n_1} \equiv \alpha_1^{n_1} \pmod{p}.
		\end{equation}
Substituting $\beta^k \equiv \frac{nc}{km-n}\pmod {p}$ into the first equation of \eqref{0.008}, we get
	\begin{align*}
		\beta^n & \equiv a^{-1} \left(\frac{kmc}{km-n}\right)^m \pmod{p},\\
		&\equiv \alpha_2 \pmod{p}.
	\end{align*}
Changing $n$ by $n_1t$ and substituting $\beta^t \equiv \alpha_1^{u_1}\alpha_2^{u_2} \pmod{p}$, we observe
	\begin{align*}
		(\alpha_1^{u_1}\alpha_2^{u_2})^{n_1} \equiv \alpha_2 \pmod{p}.
	\end{align*}
Taking the  power $k_1$ on both sides of the above equation, we get
		\begin{equation}\label{.b}
		(\alpha_1^{u_1}\alpha_2^{u_2})^{n_1k_1} \equiv \alpha_2^{k_1} \pmod{p}.
	\end{equation}
	From \eqref{.a} and \eqref{.b}, we conclude that $\alpha_1^{n_1} \equiv \alpha_2^{k_1} \pmod{p}$. This completes the proof of  the proposition.
\end{proof}

\begin{proposition} \label{4.4:prop}
    Assume the notation introduced in Proposition \ref{2.06}. Let $\a$ be an integer such that $\a \equiv \a_1^{\,u_1}\a_2^{\,u_2} \Mod{p^2}$. Let $p$ be the prime such that $p \mid (\calC - \calE)$ and $p \nmid ackm(km-n)$. Then the following are equivalent:
    \begin{itemize}
        \item[\textup{(i)}] $(\alpha^{\,k_1} + c)^m - a \alpha^{\,n_1} \equiv 0 \Mod{p^2}$;
        \item[\textup{(ii)}] $\alpha_1^{\,n_1} \equiv \alpha_2^{\,k_1} \Mod{p^2}$.
    \end{itemize}
\end{proposition}
\begin{proof}
    First, we wish to point out that both $\a_1$ and $\a_2$ are non-zero modulo $p^2$, as $p \nmid ackm(km-n)$. In addition, a simple calculation will yield $\a_1 + c \not\equiv 0 \Mod{p^2}$. Without loss of generality, assume that $u_2 \geq 0$ and $u_1 < 0$. As $p \mid (\calC-\calE)$, using Proposition \ref{2.06}, we have $\alpha_1^{\,n_1} \equiv \alpha_2^{\,k_1} \Mod{p}$. Let $\sfrak$ be an integer such that
    \begin{equation} \label{4.9:eqn}
        \a_1^{\,n_1} = \a_2^{\,k_1} + p\sfrak; \quad \quad \quad
        \a_2^{-k_1} = \a_1^{-n_1} + p\a_2^{-k_1}\a_1^{-n_1} \sfrak. 
    \end{equation}
    Using the definition of $\a_1$ and $\a_2$, a simple calculation will yield the following.
    \begin{equation} \label{4.10:eqn}
        (\a_1+c)^m - a \,\a_2 \equiv 0 \Mod{p^2}.
    \end{equation}
    In the next derivation, all the congruences are taken module $p^2$. Using $ku_1 + nu_2 = t$ and $\a \equiv \a_1^{\,u_1}\a_2^{\,u_2} \Mod{p^2}$, we have
    \begin{align*}
        (\a^{k_1} + c)^m - a \a^{n_1} 
        \equiv~ &(\a_1^{\,k_1u_1}\a_2^{\,k_1u_2} +c)^m - a \a_1^{\,n_1u_1}\a_2^{\,n_1u_2} \\
        \equiv~ &[ \a_1^{\,k_1u_1}(\a_2^{\,k_1})^{u_2} + c]^m - a (\a_1^{-n_1})^{-u_1}\a_2^{\,u_2}.
    \end{align*}
    Using \eqref{4.9:eqn}, Binomial expansion on $u_2\geq 0$ and reduction modulo $p^2$, we obtain
    \begin{align*}
        &(\a^{k_1} + c)^m - a \a^{n_1} \\
        \equiv~ &\left[ \a_1^{\,k_1u_1}\a_1^{\,n_1u_2} - u_2 \,p \,\sfrak \,\a_1^{\,k_1u_1}\,\a_1^{\,n_1(u_2-1)} + c\right]^m - a \left[ \a_2^{\,k_1u_1} + (-u_1)\, p\, \sfrak\, \a_2^{\,-k_1(-u_1-1)}\, \a_2^{-k_1} \a_1^{-n_1} \right] \a_2^{\,n_1u_2} \\
        \equiv~ &(\a_1+c + p\, \sfrak \,u_2\, \a_1\, \a_1^{-n_1})^m - a \a_2 + au_1p\,\sfrak \,\a_2 \,\a_1^{-n_1}, 
    \end{align*}
    where the last equivalence follows using the fact that $k_1 u_1 + n_1 u_2 = 1$. Applying Binomial expansion for power $m$ yields,
    \begin{align*}
        &(\a^{k_1} + c)^m - a \a^{n_1}, \\
        \equiv~ &(\a_1+c)^m + m\,u_2 \,p\, \sfrak \,\a_1^{\,1-n_1} (\a_1+c)^{m-1} - a \,\a_2 + a\,u_1\,p\,\sfrak\, \a_2 \,\a_1^{-n_1}, \\
        \equiv~ &\frac{m\,u_2\, p \,\sfrak \,\a_1^{\,1-n_1}\, a \a_2}{\a_1 + c} + a\,u_1\,p\sfrak \,\a_2\, \a_1^{-n_1},
    \end{align*}
    where the last inequality follows using \eqref{4.10:eqn} and the fact that $\a_1+c \not\equiv 0 \Mod{p^2}$. Substituting the value of $\a_1$, we get
    \begin{align}
        (\a^{k_1} + c)^m - a\, \a^{n_1}
        \equiv~ &a\,\a_2\, p \sfrak \,\a_1^{-n_1} \left[ m\,u_2 \frac{nc}{km-n} \left( \frac{kmc}{km-n} \right)^{-1} + u_1 \right], \nonumber \\
        \equiv~ &\frac{a\,\a_2 \,p \sfrak\, \a_1^{-n_1}}{k} (nu_2 + k u_1)\, \nonumber \\
        \equiv~ &p\,\sfrak\, \frac{at}{k}\, \a_2\, \a_1^{-n_1}, \label{4.11:eqn}
    \end{align}
    where the last inequality follows using the given fact $ku_1 + nu_2 = t$. Using the congruence \eqref{4.11:eqn} and the $\frac{at}{k} \a_2 \a_1^{-n_1}$ is non-zero modulo $p$, it is easy to observe that $(\a^{k_1} + c)^m - a \a^{n_1} \equiv 0 \Mod{p^2}$ if and only if $p \mid \sfrak$. This completes the proof of the proposition.
\end{proof}

\begin{proof}[Proof of Theorem~\ref{monogeneity of f(x)}]
Let $p$ be a prime that divides $\D_f$. In view of Lemma~\ref{dedekind}, we have $p \nmid [\mathbb{Z}_K : \mathbb{Z}[\theta]]$
if and only if $f(x) \not\in \langle p,\, g(x) \rangle^2$
for every monic polynomial $g(x) \in \mathbb{Z}[x]$ that is irreducible modulo $p$. Moreover, note that
$f(x) \not\in \langle p, g(x) \rangle^2$
whenever $\overline{g}(x)$ is not a repeated factor of $\bar{f}(x)$.\\[2mm]
\textbf{Case (i).} Suppose $p \mid c$ and $p \mid a$. In this case, $f(x) = (x^k + c)^m - a x^n \equiv x^{km} \pmod{p}.$ It is easily checked that $f(x) \in \langle p,\, x \rangle^2 $ if and only if $p^2 \mid c^m.$
Therefore, by Lemma~\ref{dedekind}, $p \nmid [\mathbb{Z}_K : \mathbb{Z}[\theta]]
\textit{ if and only if }
f(x) \notin \langle p,\, x \rangle^2 $ if and only if $p^2 \nmid c^m.$ Since $p \mid c$, condition $p^2 \nmid c^m$ can hold only when $m=1$ and $p^2 \nmid c$. This completes the proof in this case.\\[2mm]
\textbf{Case (ii).} Suppose that $p\mid c$, $p\nmid a$ and $n=1$.
Since $p\mid \D_f$, $p\mid c$ and $p\nmid a$, it follows from
\eqref{disc of f(x)} that $p\mid (km-1).$ Write $km-1=p^{j}s,
\, p\nmid s.$ In this situation, $f(x)\equiv x\bigl(x^{km-1}-a\bigr)
      \equiv x\bigl(x^{s}-a\bigr)^{p^{j}}
      \pmod{p}.$
Set $h(x)=x^{s}-a.$ Raising both sides to the power $p^{j}$ and using the Binomial theorem, we obtain the following.
\begin{equation}\label{eq:4.9}
x^{km-1}
= h(x)^{p^{j}} + a^{p^{j}} + p\,h(x)h_{1}(x),
\end{equation}
for some polynomial $h_{1}(x)\in\mathbb{Z}[x]$. Since $p\mid c$, we may expand $f(x)$ as
\begin{align}\label{eq:4.10}
f(x)
&= x^{km}-ax+\sum_{i=1}^{m-1}\binom{m}{i}x^{ki}c^{m-i} \nonumber \\
&= x(x^{km-1}-a)+m x^{k(m-1)}c + p^{2}h_{2}(x),
\end{align}
for some $h_{2}(x)\in\mathbb{Z}[x]$. Write $h(x)=g_{1}(x)\cdots g_{t}(x)+pH(x),$ where $g_{1}(x),\ldots,g_{t}(x)$ are distinct monic polynomials that are irreducible modulo $p$ and $H(x)\in\mathbb{Z}[x]$.
Substituting this expression into \eqref{eq:4.10} and using
\eqref{eq:4.9}, we obtain
\begin{align*}
f(x)
&= x\,h(x)^{p^{j}} + m c x^{k(m-1)} + (a^{p^{j}}-a)x
   + p\,h(x)h_{3}(x) + p^{2}h_{4}(x) \\
&= x\left(\prod_{i=1}^{t} g_i(x)+pH(x)\right)^{p^{j}}
   + x\bigl(m c x^{k(m-1)-1} + (a^{p^{j}}-a)\bigr)
   + p\,h(x)h_{3}(x) + p^{2}h_{4}(x),
\end{align*}
for some $h_{3}(x),\,h_{4}(x)\in\mathbb{Z}[x]$. Observe that $x,\,\bar g_{1}(x),\ldots,\bar g_{t}(x)$ are pairwise distinct irreducible factors of $\overline{f}(x)$.
Since $j\ge1$, it follows that $f(x)\in\langle p,\, g_i(x)\rangle^{2}$
for some $i$ if and only if the term $m c x^{k(m-1)-1} + (a^{p^{j}}-a)$ lies in $\langle p,\, g_i(x)\rangle^{2}$. Write $a^{p^{j}}-a=pa_{1},\,c=pc_{2}.$ Since $p\mid (km-1)$ implies that $p\nmid m$. Then $f(x)\in\langle p,\, g_i(x)\rangle^{2}$ for some $i$ if and only if either $p\mid c_{2}$ and $p\mid a_{1}$, or $p\nmid c_{2}$ and the polynomials $\bar m\,\bar c_{2}\,x^{k(m-1)-1}+\bar a_{1}
\quad\text{and}\quad
x^{s}-\bar a$ have a common root modulo $p$. Let $\eta$ be such a common root. Then $m c_{2}\eta^{k(m-1)-1}=-a_{1},
\,
\eta^{s}=a.$ Eliminating $\eta$, we obtain the common-root condition $a^ka_2^{km-1}=(mc)^{km-1}.$ Therefore, we conclude that $p$ does not divide $[\Z_k :\Z[\theta]]$ if and only if either $p\mid c_2$ and $p\nmid a_1$ or $p$ does not divide $c_2[a^ka_1^{\,km-1}-(-mac_2)^{\,km-1}]$. This completes the proof in the present case. \\[2mm]
\textbf{Case (iii).} Suppose that $p\mid c$ and $p\nmid a$. In this case,
$f(x)\equiv x^{n}\bigl(x^{km-n}-a\bigr)\pmod{p}.$ Since $n\ge2$, it follows that $x$ is a repeated root of $f(x)$ modulo $p$. Arguing as in Case~(i), we obtain $f(x)\notin\langle p,x\rangle^2
\quad \Longleftrightarrow \quad
p^{2}\nmid c^{m}.$ As $p\mid c$, this is possible only when $m=1$ and $p^{2}\nmid c$. Hence, we may assume throughout this case that $m=1$. Then $f(x)=x^{k}-ax^{n}+c$ and modulo $p$ we have $f(x)\equiv x^{n}\bigl(x^{k-n}-a\bigr).$ Let $\ell\ge0$ be the largest integer such that $p^{\ell}\mid(k-n)$.
\medskip
\emph{First, assume that $\ell=0$, i.e., \ $p\nmid(k-n)$.}
Since $p\mid c$ and $p\nmid a$, the reduction $\overline{f}(x)=x^{n}\bigl(x^{k-n}-\bar a\bigr)$ has $x$ as its only possible repeated irreducible factor modulo $p$. In this situation, it is easy to check that $f(x)\in\langle p,x\rangle^{2}
\textit{ if and only if }
p^{2}\mid c.$ Thus, when $\ell=0$, the desired conclusion follows immediately.\\
Now assume that $\ell\ge1$ and write $k-n=p^{\ell}s,
\quad p\nmid s.$ Set $h(x)=x^{s}-a.$ We may write $h(x)=g_{1}(x)\cdots g_{t}(x)+pH(x),$ where $g_{1}(x),\dots,g_{t}(x)$ are distinct monic polynomials that are irreducible modulo $p$ and $H(x)\in\mathbb{Z}[x]$.
Applying Lemma~\ref{binomial expansion} to $h(x)$, we obtain
\[
f(x)=x^{n}\Biggl[
\left(\prod_{i=1}^{t} g_i(x)+pH(x)\right)^{p^{\ell}}
+pT(x)\prod_{i=1}^{t} g_i(x)
+p^{2}U(x)
-a+a^{p^{\ell}}
\Biggr]+c,
\]
for some polynomials $T(x),U(x)\in\mathbb{Z}[x]$. Observe that $x,\bar g_{1}(x),\ldots,\bar g_{t}(x)$ are pairwise distinct
irreducible factors of $\overline{f}(x)$.
Since $\ell\ge1$, the first three terms inside the brackets are in
$\langle p,g_i(x)\rangle^{2}$ for each $1\le i\le t$.
Consequently, $f(x)\in\langle p,g_i(x)\rangle^{2}$ for some $i$ if and only if the remaining term $\bigl(a^{p^{\ell}}-a\bigr)x^{n}+c$
lies in $\langle p,\, g_i(x)\rangle^{2}$. Write $a^{p^{\ell}}-a=pa_{1},
\, c=pc_{2}.$
Then $f(x)\in\langle p,\,g_i(x)\rangle^{2}$ for some $i$ if and only if
either $p\mid a_{1}$ and $p\mid c_{2}$, or $p\nmid a_{1}$ and the
polynomials $\bar a_{1}x^{n}+\bar c_{2}
\quad\text{and}\quad
x^{k-n}-\bar a$ have a common root modulo $p$. Let $\gamma$ be such a common root. Then
\[
a_{1}\gamma^{n}+c_{2}=0,
\qquad
\gamma^{k}=a\gamma^{n}.
\]
Simplifying these relations and using $k=k_{1}t$ and $n=n_{1}t$, we obtain
\[
(-c_{2})^{\,k_{1}-n_{1}}=a^{\,n_{1}}a_{1}^{\,k_{1}-n_{1}}.
\]
Hence, the above polynomials have a common root modulo $p$ if and only if
\[
(-c_{2})^{\,k-n}=a^{\,n}a_{1}^{\,k-n}.
\]
By Lemma~\ref{dedekind}, we conclude that
$p\nmid[\mathbb{Z}_{K}:\mathbb{Z}[\theta]]$ if and only if either
$p\mid a_{1}$ and $p\nmid c_{2}$, or
$
p\nmid a_{1}c_2
\bigl[a^{n_{1}}a_{1}^{\,k_{1}-n_{1}}-(-c_{2})^{\,k_{1}-n_{1}}\bigr].$
This completes the proof in the present case.\\[2mm]
\textbf{Case (iv).} Suppose that $p\nmid c$ and $p\mid a$. Then $f(x) \equiv (x^{k}+c)^{m} \pmod{p}.$ There are two possibilities to consider, namely $m\ge 2$ and $m=1$.\\
\emph{First, assume that $m\ge 2$.}
Let $g(x)$ be an irreducible factor of $x^{k}+c$ over $\mathbb{Z}/p\mathbb{Z}$.
Then $f(x)\in\langle p,g(x)\rangle^{2}$ if and only if $p^{2}\mid a$.
Hence, by Lemma~\ref{dedekind}, we obtain $p\nmid [\mathbb{Z}_{K}:\mathbb{Z}[\theta]]
\quad \Longleftrightarrow \quad
p^{2}\nmid a.$\\
\emph{Now assume that $m=1$.}
In this case $f(x)=x^{k}-ax^{n}+c,$ and since $p\mid a$, we have $f(x)\equiv x^{k}+c \pmod{p}.$ Moreover, as $p\mid \D_f$, it follows from \eqref{disc of f(x)} that $p\mid k$. Write $k=p^{\ell}s \, \text{ with }\, p\nmid s.$ By the Binomial theorem, we then obtain $f(x)\equiv (x^{s}+c)^{\,p^{\ell}} \pmod{p}.$ Let $\bar g_{1}(x),\ldots,\bar g_{r}(x)$ be the factorization of
$x^{s}+\bar c$ over $\mathbb{Z}/p\mathbb{Z}$, where
$g_i(x)\in\mathbb{Z}[x]$ are monic, pairwise distinct and irreducible
modulo $p$. We may therefore write $x^{s}+c = g_{1}(x)\cdots g_{r}(x) + pH(x)$ for some polynomial $H(x)\in\mathbb{Z}[x]$. Set $h(x)=x^{s}+c$.
Then $f(x)=h(x^{p^{\ell}})-ax^{n}.$ Applying Lemma~\ref{binomial expansion} to $h(x)$, we obtain
\begin{equation}\label{eq:case-iii-expansion}
f(x)= \left(\prod_{i=1}^{r} g_i(x) + pH(x)\right)^{p^{\ell}}
   + pT(x)\prod_{i=1}^{r} g_i(x)
   + p^{2}U(x) + c + (-c)^{p^{\ell}} - ax^{n},
\end{equation} 
for some polynomials $T(x),U(x)\in\mathbb{Z}[x]$.
Since $\ell\ge 1$, the first three terms on the right-hand side of
\eqref{eq:case-iii-expansion} lie in the ideal
$\langle p, g_i(x)\rangle^{2}$ for each $1\le i\le r$.
Therefore,
$f(x)\in\langle p, g_i(x)\rangle^{2}$ for some $i$
if and only if $-ax^{n}+c+(-c)^{p^{\ell}} = p(a_{2}x^{n}+c_{1})$
belongs to $\langle p, g_i(x)\rangle^{2}$. It follows that
$p(a_{2}x^{n}+c_{1})\in\langle p, g_i(x)\rangle^{2}$
for some $i$ if and only if either $p\mid a_{2}$ and $p\mid c_{1}$,
or $p\nmid a_{2}$ and the polynomials
$\overline{a}_{2}x^{n}+\overline{c}_{1}$ and $x^{s}+\overline{c}$
have a common root modulo $p$. Let $\gamma$ be such a common root.
Then $a_{2}\gamma^{n}+c_{1}=0
\quad \text{and} \quad
\gamma^{k}+c=0.$ Writing $k=k_{1}t$ and $n=n_{1}t$, we obtain
\[
a_{2}(\gamma^{t})^{n_{1}}=-c_{1}
\quad \text{and} \quad
(\gamma^{t})^{k_{1}}=-c.
\]
The rise of the first equation to the power $k_{1}$ and the second to the power $n_{1}$ yields the following equation.
\[
a_{2}^{k_{1}}(\gamma^{t})^{n_{1}k_{1}}
=(-c_{1})^{k_{1}}
\quad \text{and} \quad
(\gamma^{t})^{k_{1}n_{1}}
=(-c)^{n_{1}}.
\]
Comparing these expressions, we conclude that $a_{2}^{k_{1}}(-c)^{n_{1}} = (-c_{1})^{k_{1}}.$ Therefore, the polynomials
$\overline{a}_{2}x^{n}+\overline{c}_{1}$ and $x^{s}+\overline{c}$
are coprime modulo $p$ if and only if $p\nmid \bigl(a_{2}^{k_{1}}(-c)^{n_{1}}-(-c_{1})^{k_{1}}\bigr).$ Therefore, by Lemma~\ref{dedekind}, we obtain $p\nmid [\mathbb{Z}_{K}:\mathbb{Z}[\theta]]$ if and only either $p\mid a_2$ and $p\nmid c_1$ or $p\nmid a_2\bigl[a_{2}^{k_{1}}(-c)^{n_{1}}-(-c_{1})^{k_{1}}\bigr]$.\\[2mm]
\textbf{Case (v).} Suppose $p \nmid ca$ and $p \mid k$. By Theorem~\ref{disc of f(x)}, we have 
$p \mid (\mathcal{C} - \mathcal{E})$, and hence $p \mid n(km - n).$
Since $p \nmid a$ and $p \nmid c$, we obtain $p \mid n$. Write $k = p^{j}s' \quad \text{and} \quad n = p^{j}s,$ with $j \ge 1$ and $p \nmid \gcd(s,s')$. Then
\[
f(x) \equiv \Bigl( (x^{s'} + c)^m - a x^s \Bigr)^{p^{j}} \pmod{p}.
\]
 Denote  $ (x^{s'} + c)^m - a x^s$ by $h(x)$ so that $f(x)=h(x^{p^j})$. Then one can easily check that $f(x)\equiv h(x)^{p^j} \pmod{p}$. Write $h(x) = g^{e_1}_1(x)\cdots g^{e_d}_d(x) + ph_1(x)$, $e_i>0$, where $g_1(x), \ldots,$ $g_d(x)$ are monic polynomials which are distinct as well as irreducible modulo $p$ and $h_1(x) \in \Z[x]$. Using Lemma \ref{lemma:3.6}, we have
\begin{align*}
    f(x) &=  h(x)^{\,p^{j}}+(a^{\,p^{j}} - a)\,x^{n} - p m\, t(x)(x^{k} + c)^{\,m-1}+ p\,h(x)\,h_{1}(x) + p^{2}\,h_{4}(x)\\
    &=\left(\prod\limits_{i=1}^{d}g^{e_i}_i(x)\right)^{\,p^j}+(a^{\,p^{j}} - a)\,x^{n} - p m\, t(x)(x^{k} + c)^{\,m-1}+ p\,\left(\prod\limits_{i=1}^{d}g^{e_i}_i(x)\right)\,h_{1}(x) + p^{2}\,h_{4}(x)
\end{align*}
for some polynomials $h_1(x),\,h_4(x)\in\Z[x]$. Write $f(x)=\left(g_1(x)^{e_1} \cdots g_d(x)^{e_d}\right)^{p^j}+p M(x)$ for some $M(x)\in\Z[x].$ Since  $j>0$, by Lemma \ref{dedekind}, we see that $p$ does not divide $[\Z_K :\Z[\theta]]$ if and only if $\overline{M}(x)$ is coprime to $\overline{h}(x)$, which holds if and only if the polynomial $\frac{1}{p}[(a^{\,p^{j}} - a)\,x^{n} - p m\, t(x)(x^{k} + c)^{\,m-1}]$ is coprime to $h(x)$ modulo $p$. This proves the theorem in this case.\\[2mm]
\textbf{Case (vi).} Suppose $p \nmid cak$ and $p \mid m$. By Theorem~\ref{disc of f(x)}, we have $p \mid n$. Write $n = p^{\ell}s \quad \text{and} \quad m = p^{\ell}s',$ where $\ell \ge 1$ and $p \nmid \gcd(s,s')$. Define $h(x) = (x^{k}+c)^{s'} - a x^{s}.$ Then it is easy to check that $f(x) \equiv h(x)^{\,p^{\ell}} \pmod{p}.$
We first expand $(h(x)+a x^{s})^{p^{\ell}}$ using the binomial theorem:
\[
(h(x)+a x^{s})^{p^{\ell}} = h(x)^{p^{\ell}} + pH(x) + a^{p^{\ell}}x^{n},
\]
for some polynomial $H(x)\in\mathbb{Z}[x]$. Since $h(x)=(x^{k}+c)^{s'}-a x^{s}$, it follows that
\[
(x^{k}+c)^{m}
= h(x)^{p^{\ell}} + p\,h(x)H(x) + (a^{p^{\ell}}-a+a)x^{n}.
\]
Consequently, we may write the following.
\begin{equation}\label{eq:5.1}
f(x) = h(x)^{p^{\ell}} + p\,h(x)h_{1}(x) + (a^{p^{\ell}}-a)x^{n},
\end{equation}
for some polynomial $h_{1}(x)\in\mathbb{Z}[x]$. Next, write $h(x)=g_{1}(x)^{e_{1}}\cdots g_{d}(x)^{e_{d}} + p\,h_{2}(x),
\quad e_i>0,$ where $g_{1}(x),\ldots,g_{d}(x)$ are distinct monic polynomials that are irreducible modulo $p$ and $h_{2}(x)\in\mathbb{Z}[x]$. Substituting this expression into \eqref{eq:5.1}, we obtain
\begin{align*}
f(x)
&= \left(\prod_{i=1}^{d} g_{i}(x)^{e_{i}} + p\,h_{2}(x)\right)^{p^{\ell}}
   + p\,h(x)h_{1}(x) + (a^{p^{\ell}}-a)x^{n} \\
&= \left(\prod_{i=1}^{d} g_{i}(x)^{e_{i}}\right)^{p^{\ell}}
   + p\,h(x)h_{1}(x) + p^{2}h_{3}(x) + (a^{p^{\ell}}-a)x^{n},
\end{align*}
for some polynomial $h_{3}(x)\in\mathbb{Z}[x]$. Thus, we may write 
$
f(x)=\left(g_{1}(x)^{e_{1}}\cdots g_{d}(x)^{e_{d}}\right)^{p^{\ell}} + pM(x),
$ for some $M(x)\in\mathbb{Z}[x]$. Since $\ell>0$, Lemma~\ref{dedekind} implies that $p \nmid [\mathbb{Z}_{K}:\mathbb{Z}[\theta]]$ if and only if $\overline{M}(x)$ is co-prime to $\overline{h}(x)$ modulo $p$.
This holds if and only if the polynomial $a_1 x^{n}$ is co-prime to $h(x)$ modulo $p$,
which is equivalent to condition $p\nmid a_1$, where $a_1=\frac{a^{p^{\ell}}-a}{p}.$ This condition also simplifies to $p^{2}\nmid (a^{p}-a)$. This completes the proof in this case.\\[2mm] 
\textbf{Case (vii).} Now consider the last  case where $p \nmid ackm $. Furthermore, as $p\mid \D_f$ and $p\nmid km$ and $k_1m-n_1 \neq 0$, it follows that $p \nmid n(km-n)$. We claim that there exists an integer $\alpha$ such that the polynomial $x^t - \bar{\alpha}$ is the product of all distinct monic repeated irreducible factors of $\bar{f}(x)$ over $\Z/p\Z$. 
	
	Let $\alpha_1$ and $\alpha_2$ be as defined in Proposition \ref{2.06} and $\alpha = \alpha_1^{u_1} \alpha_2^{u_2} \pmod{p^2}$. Suppose $\beta$ is any repeated root of $\bar f (x)= (x^k+\bar c)^m -\bar a x^n$ in algebraic closure of $\Z/p\Z.$ Then using the derivation similar to those for \eqref{0.009}, we have $\beta^k \equiv \alpha_1 \pmod{p}$. Clearly, $\beta^k \in \Z/p\Z$ and $f(\beta)\equiv 0 \pmod{p}$ imply that
	\begin{align*}
		\beta^n \equiv a^{-1} (\beta^k +c)^m \equiv \alpha_{2} \in {\Z/p\Z}.
	\end{align*}
	As ${\gcd(n,k)}=t$ so there exist $u_1,u_2\in\Z$ such that $ku_1+nu_2=t.$ Using this we obtain $\beta^t = (\beta^k)^{u_1}(\beta^n)^{u_2} \in {\Z/p\Z}$, i.e., $\beta^t \equiv \alpha_1^{u_1} \alpha_{2}^{u_2}\pmod{p}$. Thus, we have proved that any repeated root of $\bar f(x)$ is a root of $ x^t-\bar{\alpha}$. 
	Now using Proposition \ref{2.06}, if $\beta_1$ is a root of $x^t-\bar{\alpha}$  with $\beta_1 \equiv \alpha_1^{u_1} \alpha_{2}^{u_2} \pmod{p}$, then $\beta_1$ is a repeated root of $\bar{f}(x)$ if and only if $p\mid (\calC - \calE)$. Hence, we conclude that every root of $x^t-\bar{\alpha}$ is a repeat root of $f(x)$ modulo $p$. We have shown that every root of $x^t - \bar{\alpha}$ is a repeated root of $f(x)$ modulo $p$.  
%Since $\deg{g(x)}=t\geq 1$, hence $g(x)$ has a root in algebraic closure of $\Z/p\Z$. Let $\beta_0$  be a root of $g(x)$ in algebraic closure of $\Z/p\Z$. 
Since $t=\gcd(n,k)$, it is easy to see that
	\begin{equation}\label{3.20}
		f(x)= ((x^t)^{k_1}+c)^m -a(x^t)^{n_1} = (x^t-\alpha)q(x) +  (\alpha^{k_1}+c)^m - a \alpha^{n_1},
	\end{equation}
	for some $q(x) \in \Z[x^t]$. As $x^t - \bar{\alpha}$ divides $\bar{f}(x)$, we have $\bar{f}(x) = (x^t-\bar{\alpha})\bar{q}(x)$. Let $\bar{f}(x)=\bar g_1(x)^{e_1}\cdots \bar g_t(x)^{e_t}$ be the factorization of $\bar{f}(x)$ into a product of  powers of distinct irreducible polynomials over $\Z /p\Z$ with each $g_{i}(x)\in Z[x]$ monic. If necessary, after renaming assume that $e_{i}>1$ for $1\leq i\leq t_{1}$ and $e_i=1$ for $t_1< i \leq t$. Then  $x^t-\bar{\alpha}=\prod\limits_{i=1}^{t_1}\bar g_i(x)$. Write\\\\
	$x^t-\alpha=\prod\limits_{i=1}^{t_1} g_i(x)+ph_1(x), \quad q(x)=\prod\limits_{i=1}^{t_1}g_i(x)^{e_i-1}~~ \prod\limits_{i=t_1+1}^{t}g_i(x)+ph_2(x)$ \\\\ for some $h_1(x),h_2(x)\in\Z[x]$. Substituting from the above equation in $\eqref{3.20}$, we  have
	\begin{align*}
		f(x)&=\prod_{i=1}^{t}g_i(x)^{e_i}
+ph_1(x)\prod_{i=1}^{t_1}g_i(x)^{e_i-1}	\prod_{i=t_1+1}^{t}g_i(x)+ph_2(x)\prod_{i=1}^{t_1}g_i(x) \nonumber\\ &+~p^2h_1(x)h_2(x) + (\alpha^{k_1}+c)^m - a \alpha^{n_1}.
	\end{align*}
 Clearly, each summand on the right-hand side of the above equation except possibly $(\alpha^{k_1}+c)^m - a \alpha^{n_1}$ belongs to $\langle p, g_i(x) \rangle^2$  for $1\leq i\leq t_1$. So $f(x)$ $\in \langle p, g_i(x) \rangle^2$ for some $i$ $1\leq i\leq t_1$ if and only if $p^2$ divides $(\alpha^{k_1}+c)^m - a \alpha^{n_1}$. As $\bar f(x)$ is not divisible by $\bar g_i(x)^2$ for $t_1< i\leq t$, it is clear that $f(x)\not\in  \langle p, g_i(x) \rangle^2$ for such $i$. So $f(x)\not \in  \langle p, g_i(x) \rangle^2$ for any $i$, $1 \leq i \leq t$ if and only if $p^2\nmid ((\alpha^{k_1}+c)^m - a \alpha^{n_1})$. To prove this case, it only remains to prove that $(\alpha^{k_1}+c)^m - a \alpha^{n_1} \equiv 0 \pmod{p^2}$ if and only if $\calC - \calE \equiv 0\pmod{p^2}$. This follows from Propositions \ref{2.06} and \ref{4.4:prop} and hence the theorem.
\end{proof}
\begin{proof}[Proof of Corollary~\ref{6.1lemma}]
Let $f(x) = (x^{q^j} - 1)^m - a x^n$, and let $t$ be an integer such that $t \equiv 1 \pmod q$. Consider the shifted polynomial
$f(x+t) = ((x+t)^{q^j} - 1)^m - a(x+t)^n = \sum_{i=0}^{mp^j} a_i x^i$. The constant term of $f(x+t)$ is
$a_0 = (t^{q^j} - 1)^m - a t^n$. Since $t \equiv 1 \pmod q$, we have $t^{q^j} \equiv 1 \pmod q$, and therefore $v_q(t^{q^j} - 1) \ge 1$. It follows that $v_q((t^{q^j} - 1)^m) \ge m \ge 2$.
Moreover, by expanding $(x+t)^{q^j}$ and $(x+t)^n$ using the Binomial theorem, and using the fact that $q$ divides all intermediate binomial coefficients of $(x+t)^{q^j}$, one checks that $q$ divides the coefficient of all $x^j$ for $0 \leq j < mq^j$. Since $q \mid\mid a$, we have $v_q(a) = 1$, and because $q \nmid t$, it follows that $v_q(a t^n) = 1$. Therefore, $v_q(a_0) = v_q((t^{p^j} - 1)^m - a t^n) = 1$. Hence, $f(x+t)$ is $q$-Eisenstein, and consequently $f(x)$ is irreducible. Similarly, $f(x) = (x^{q^j} + 1)^m - a x^n$ is also irreducible. Now we address the monogeneity of $f(x)$. Since $c=\pm 1$, the first three cases of Theorem~\ref{monogeneity of f(x)} do not occur. As $m \geq 2$, case~(iv)(b) of Theorem~\ref{monogeneity of f(x)} also does not occur. Furthermore, the polynomial under consideration has $k=q^j$ and $q$ is a divisor of $a$, case~(v)
of Theorem~\ref{monogeneity of f(x)} also does not occur. This completes the proof.
\end{proof}

\section[Monogeneity of Composition of f(x) with Arbitrary Polynomial g(x)]{Monogeneity of Composition of \texorpdfstring{$f(x)$}{f(x)} with Arbitrary Polynomial \texorpdfstring{$g(x)$}{g(x)}}
To determine all possible primes that may divide the index of $f \circ g$, 
we first compute its discriminant.

Note that the discriminant of a monic polynomial $f(x) \in \mathbb{Z}[x]$ of degree $n$ 
is related to the resultant\footnote{\noindent If $\lambda_1,\ldots,\lambda_n$ and
$\mu_1,\ldots,\mu_m$ are the roots of $f$ and $g$, respectively, then the
resultant of $f$ and $g$ is given by
\[
\Res(f,\,g)
=
a^m b^n
\prod_{\substack{1 \le i \le n \\ 1 \le j \le m}}
(\lambda_i-\mu_j),
\] where $a$ and $b$ are the leading coefficients of $f$ and $g$, respectively.} $R(f, f')$ by the formula
\[
D_f = (-1)^{\frac{n(n-1)}{2}} R(f, f'),
\]
where $f^{\prime}(x)$ denotes the derivative of $f(x)$. 
\begin{proposition}\label{prop:disc_comp}
Let $f(x)$ and $g(x)$ be polynomials of degrees $n$ and $m$, respectively, 
with leading coefficients $a$ and $b$. Then
\begin{equation}\label{disc f of g}
\D_{f \circ g} 
= \pm\, a^{\,m-1}\, b^{\,n(mn - m - 1)} \,
\D_f^{m} \,
\operatorname{Res}\bigl(f\circ g,\, g^{\prime}\bigr).
\end{equation}
\end{proposition}
To prove Theorem~\ref{monogeneity of f(g(x))}, we first establish the following proposition, which characterizes the prime divisors of \( [\mathbb{Z}_K : \mathbb{Z}[\theta]] \) for composition $f\circ g$.
\begin{proposition}\label{2.5:pro}
    Let $f(x) = (x^k+c)^m-ax^n\in \mathbb{Z}[x]$ be irreducible over $\mathbb{Q}$ and let  $g(x)\in \Z[x]$ be a polynomial of degree $d$ such that \(T(x) := f(g(x))\) is irreducible over $\mathbb{Q}$.  Let $K = \mathbb{Q}(\theta)$, where $\theta$ is a root of $T(x)$. Assume that $\rad(k)\mid ca$ and that, for every prime $p \in \scrE_f$, we have $p \mid g(0)$, $p \nmid d$, and the polynomial $g(x)$ satisfies \( g'(x) \equiv d x^{d-1} \pmod{p}.\) Furthermore, in the case $n = 1$, we assume that for every prime $p$ satisfying 
$p \mid \D_f$, $p \mid c$, and $p \nmid a$, one has $p \nmid \D_g$.
 Then, the following statements are equivalent:
    \begin{enumerate}[label=\textup{(\roman{enumi})}]
    \item  If $p\mid c$  and $p\mid a$, then $m=1$ and $p^2\nmid c$;
    \item  If $p\mid c,~~p\nmid a$ and $n=1$  with $j\geq 1$ as the highest power of $p$ dividing $km-1$, then either $p\mid c_2$ and $p\nmid a_1$ or $p\nmid c_2[a^ka_1^{\,km-1}-(-mac_2)^{\,km-1}]$, where $c_2=\frac{c}{p}$ and $a_1=\frac{a^{p^j}-a}{p}$;
    \item  If $p\mid c,~~p\nmid a$ and $n\geq 2$ with $\ell\geq 0$ as the highest power of $p$ dividing $k-n$, then $m=1$ and  either
$p\mid a_{1}$ and $p\nmid c_{2}$, or $p\nmid a_{1}c_{2}^{\,n-1}
\bigl[a^{n_1}a_{1}^{\,k_1-n_1}-(-c_{2})^{\,k_1-n_1}\bigr]$, where $c_2=\frac{c}{p}$ and $a_1=\frac{a^{\,p^{\ell}}-a}{p}$;
    \item  If $p\nmid c$ and $p\mid a$ and further;
    \begin{itemize}
        \item[\textup{(a)}] If $m\geq 2$, then $p^2\nmid a$;
        \item [\textup{(b)}] If $m=1$ with $\ell\geq 1$ as the highest power of $p$ dividing $k$, then either $p\mid a_2$ and $p\nmid c_1$ or $p\nmid a_2\bigl[a_{2}^{k_{1}}(-c)^{n_{1}}-(-c_{1})^{k_{1}}\bigr]$, where $a_2=-\frac{a}{p}$ and $c_1=\frac{(-c)^{\,p^{\ell}}+c}{p}$;
    \end{itemize}
     \item If $p\nmid cak$ and $p\mid m$, then $p \mid n$ and $p^2\nmid (a^{p}-a)$;
    \item  If $p\nmid cakm$, then $p^2\nmid (\calC-\calE).$
	\end{enumerate}
\end{proposition}
\begin{proof}
\textbf{Case (i).} Suppose that $p\mid c$ and $p\mid a$. In this case,
$T(x)\equiv g(x)^{km}\pmod{p}$. Let $h(x)$ be an irreducible factor of
$g(x)$ modulo $p$. Since $km>1$, one easily checks that $T(x)\in\langle p,h(x)\rangle^{2}$ if and only if $p^{2}\mid c^{m}.$ Hence, by Lemma~\ref{dedekind}, we have
\[
p\nmid[\mathbb{Z}_{K}:\mathbb{Z}[\theta]]
\quad \text{if and only if} \quad
p^{2}\nmid c^{m}.
\]
Since $p\mid c$, the condition $p^{2}\nmid c^{m}$ can occur only when
$m=1$ and $p^{2}\nmid c$.\\[2mm]
\textbf{Case (ii).} Suppose that $p\mid c$, $p\nmid a$, and $n=1$.
Since $p\mid \D_f$, $p\mid c$, and $p\nmid a$, it follows from
\eqref{disc of f(x)} that $p\mid (km-1)$.
Write $km-1=p^{j}s$ with $p\nmid s$.
In this situation,
\[T(x)\equiv g(x)\bigl(g(x)^{km-1}-a\bigr)
\equiv g(x)\bigl(g(x)^{s}-a\bigr)^{p^{j}} \pmod{p}.\]
By hypothesis, $p\nmid \D_g$; hence any repeated roots of $T(x)$
modulo $p$ can arise only from repeated factors of $g(x)^{km-1}-a$
modulo $p$.
Set $h(x)=g(x)^{s}-a$.
Raising both sides to the power $p^{j}$ and applying the Binomial
theorem, we obtain
$g(x)^{km-1}=h(x)^{p^{j}}+a^{p^{j}}+p\,h(x)h_{1}(x)$
for some $h_{1}(x)\in\mathbb{Z}[x]$. Arguing as in Theorem~\ref{monogeneity of f(x)} \textup{(ii)}, we conclude that
$p\nmid[\mathbb{Z}_{K}:\mathbb{Z}[\theta]]$ if and only if either
$p\mid c_{2}$ and $p\nmid a_{1}$, or
$p\nmid c_{2}\bigl[a^{k}a_{1}^{\,km-1}-(-ma c_{2})^{\,km-1}\bigr]$.
This completes the proof in the present case.\\[2mm]
\textbf{Case (iii).} Suppose that $p\mid c$ and $p\nmid a$.
In this case,
$T(x)\equiv g(x)^{n}\bigl(g(x)^{km-n}-a\bigr)\pmod{p}$.
Since $n\ge2$, every root of $g(x)$ is a repeated root of $T(x)$ modulo $p$.
Let $h_{1}(x)$ be a root of $g(x)$ modulo $p$.
Arguing as in Case~(i), we obtain $f(x)\notin\langle p,h_{1}(x)\rangle^{2}
 \text{ if and only if }  p^{2} \nmid c^{m}.$
As $p\mid c$, this can occur only when $m=1$ and $p^{2}\nmid c$.\\
Hence, throughout this case we may assume that $m=1$. Then $T(x)=g(x)^{k}-ag(x)^{n}+c$ and modulo $p$ we have
$T(x)\equiv g(x)^{n}\bigl(g(x)^{k-n}-a\bigr)$.
Let $\ell\ge0$ be the largest integer such that $p^{\ell}\mid(k-n)$.\\[1mm]
First assume that $\ell=0$, i.e.\ $p\nmid(k-n)$. In this case, we have $p \in \scrE_f$. The assumption on $g$ will imply that $g(x)^{k-n}-a$ will not have any repeated root modulo $p$.
Thus, the reduction $\overline{T}(x)=g(x)^{n}\bigl(g(x)^{k-n}-\bar a\bigr)$ has repeated roots only coming from $g(x)$. Let $h_{1}(x)$ be a root of $g(x)$ modulo $p$. Then it is easy to check that $T(x)\in\langle p,h_{1}(x)\rangle^{2}
\text{ if and only if } p^{2}\mid c.$  Thus, when $\ell=0$, the desired conclusion follows immediately.\\[1mm]
%Check from the previous paper how we handled this type of case to avoid any repeated factor
Now assume that $\ell\ge1$ and write $k-n=p^{\ell}s$ with $p\nmid s$.
Set $h(x)=g(x)^{s}-a$.
Arguing as in Theorem~\ref{monogeneity of f(x)} {(iii)}, we conclude that
$p\nmid[\mathbb{Z}_{K}:\mathbb{Z}[\theta]]$ if and only if either
$p\mid a_{1}$ and $p\nmid c_{2}$, or $p\nmid a_{1}c_{2}^{\,n-1}
\bigl[a^{n_{1}}a_{1}^{\,k_{1}-n_{1}}-(-c_{2})^{\,k_{1}-n_{1}}\bigr],$ where $c_{2}=c/p$ and $a_{1}=(a^{p^{\ell}}-a)/p$. This completes the proof in the present case.
\\[2mm]
\textbf{Case (iv).} Suppose that $p\nmid c$ and $p\mid a$.
Then $T(x)\equiv (g(x)^{k}+c)^{m}\pmod{p}.$ There are two possibilities to consider, namely $m\ge2$ and $m=1$. First consider that $m\geq 2$. Let $h(x)$ be an irreducible factor of $g(x)^{k}+c$ over
$\mathbb{Z}/p\mathbb{Z}$.
Then $T(x)\in\langle p,h(x)\rangle^{2}
\textit{ if and only if }
p^{2}\mid a.$
Hence, by Lemma~\ref{dedekind}, we obtain $p\nmid[\mathbb{Z}_{K}:\mathbb{Z}[\theta]]
\textit{ if and only if }
p^{2}\nmid a.$\\[1mm]
Now assume that $m=1$. In this case, $T(x)=g(x)^{k}-ag(x)^{n}+c,$ and since $p\mid a$, we have
$T(x)\equiv g(x)^{k}+c \pmod{p}.$ Moreover, as $p\mid \D_f$, it follows from \eqref{disc of f(x)} that
$p\mid k$.
Write $k=p^{\ell}s \quad \text{with} \quad p\nmid s.$ By the Binomial theorem, we then obtain $T(x)\equiv (g(x)^{s}+c)^{\,p^{\ell}} \pmod{p}.$
Arguing as in Theorem~\ref{monogeneity of f(x)} \textup{(iv)}, we conclude that
$p\nmid[\mathbb{Z}_{K}:\mathbb{Z}[\theta]]$ if and only if either $p\mid a_2$ and $p\nmid c_1$ or $p\nmid a_2\bigl[a_{2}^{k_{1}}(-c)^{n_{1}}-(-c_{1})^{k_{1}}\bigr]$.
\\[2mm]

\textbf{Case (v).} Suppose that $p\nmid cak$ and $p\mid m$.
By Theorem~\ref{disc of f(x)}, we have $p\mid n$.
Write $n=p^{\ell}s \quad \text{and} \quad m=p^{\ell}s',$ where $\ell\ge1$ and $p\nmid \gcd(s,s')$.
Define $h(x)=(g(x)^{k}+c)^{s'}-a g(x)^{s}.$ Arguing as in Theorem~\ref{monogeneity of f(x)} \textup{(vi)}, we conclude that
$p\nmid[\mathbb{Z}_{K}:\mathbb{Z}[\theta]]$ if and only if
$p^{2}\nmid (a^{p}-a)$.
This completes the proof in the present case.
\\[2mm]
\textbf{Case (vi).} Now consider the last  case when $p \nmid ackm $. Furthermore, as $p\mid \D_f$ and $p\nmid km$ and $k_1m-n_1 \neq 0$, it follows that $p \nmid n(km-n)$. We claim that there exists an integer $\alpha$ such that the polynomial $g(x)^t - \bar{\alpha}$ is the product of all the distinct monic repeated irreducible factors of $\overline{T}(x)$ over $\Z/p\Z$. 
Let $\alpha_1$ and $\alpha_2$ be as defined in Proposition \ref{2.06} and $\alpha = \alpha_1^{u_1} \alpha_2^{u_2} \pmod{p^2}$. Suppose $\beta$ is any repeated root of $\bar T (x)= (g(x)^k+\bar c)^m -\bar a g(x)^n$ in algebraic closure of $\Z/p\Z.$ In this case, $p \in \scrE_f$. Then using the assumption on $g(x)$ modulo $p$ and the derivation similar to those for \eqref{0.009}, we have $g(\beta)^k \equiv \alpha_1 \pmod{p}$. Clearly, $g(\beta)^k \in \Z/p\Z$ and $f(\beta)\equiv 0 \pmod{p}$ implies that
	\begin{align*}
		g(\beta)^n \equiv a^{-1} (g(\beta)^k +c)^m \equiv \alpha_{2} \in {\Z/p\Z}.
	\end{align*}
As ${\gcd(n,k)}=t$ so there exist $u_1,u_2\in\Z$ such that $ku_1+nu_2=t.$ Using this we obtain $g(\beta)^t = (\beta^k)^{u_1}(\beta^n)^{u_2} \in {\Z/p\Z}$, i.e., $g(\beta)^t \equiv \alpha_1^{u_1} \alpha_{2}^{u_2}\pmod{p}$. Thus, we have proved that any repeated root of $\bar T(x)$ is a root of $ g(x)^t-\bar{\alpha}$. Now using arguments similar to those used in the proof of Proposition \ref{2.06}, we obtain that if $\beta_1$ is a root of $g(x)^t-\bar{\alpha}$  with {$g(\beta_1) \equiv \alpha_1^{u_1} \alpha_{2}^{u_2} \pmod{p}$}, then $\beta_1$ is repeated root of $\bar{T}(x)$ if and only if $p\mid (\calC - \calE)$. Hence, we conclude that every root of $g(x)^t-\bar{\alpha}$ is repeated root of $T(x)$ modulo $p$. We have proved that every root of $g(x)^t - \bar{\alpha}$ is a repeated root of $T(x)$ modulo $p$. Since $t=\gcd(n,k)$, it is easy to see that
	\begin{equation}\label{3.20b}
		T(x)= ((g(x)^t)^{k_1}+c)^m -a(g(x)^t)^{n_1} = (g(x)^t-\alpha)q(x) +  (\alpha^{k_1}+c)^m - a \alpha^{n_1},
	\end{equation}
	for some $q(x) \in \Z[x^t]$.  Arguing as in Theorem~\ref{monogeneity of f(x)} \textup{(vii)}, we conclude that
$p\nmid[\mathbb{Z}_{K}:\mathbb{Z}[\theta]]$ if and only if
$p^{2}\nmid (\calC-\calE)$.
This completes the proof in the present case.
\end{proof}
\begin{proof}[Proof of Theorem~\ref{monogeneity of f(g(x))}]
  The proof follows easily from Proposition~\ref{2.5:pro}.
\end{proof}
\section{Analytic Results Related to Monogeneity of \texorpdfstring{$f(x)$}{f(x)}}\label{ana}
This section is devoted to analytic results concerning the monogeneity of the polynomial $f(x)$. Our approach relies crucially on a theorem of Jones and White~\cite[Theorem 3.8]{L3}, which serves as the main tool in establishing an asymptotic formula for the number of monogenic polynomials in the given family.\\
 Consider an integer $\beta$ and positive integers $\rho, \gamma, \alpha, \alpha_0$ and $\beta_0$ that satisfy the following:
    \begin{empheq}[box=\fbox]{equation} \label{6.1:eqn}
        \begin{aligned}
            \text{(i)} ~&\gcd(\alpha_0 \beta_0 \rho, \gamma)=1= \gcd(\alpha, \beta), \\
            \text{(ii)} ~&\text{for each prime } p \mid \beta, \text{ we also have } p^2 \mid \beta, \\
            \text{(iii)} ~&\ao, \bo \text{ be squarefree divisors of } \alpha, \beta \text{ respectively}, \\
            \text{(iv)} ~&\alpha \beta_0 \rho + \beta \not\equiv 0 \pmod{p^2} \text{ for every } p \mid \gamma.
        \end{aligned}
    \end{empheq}
    \noindent Using the above notations, Jones and White \cite[Theorem 3.8]{L3} proved an asymptotic result for the set defined by
    \begin{align*}
        U(X)=U(X; \rho,\gamma, \alpha, \alpha_0, \beta, \beta_0) := |\{  y \leq X: y \equiv \rho~ (\text{mod } \gamma^2),\, \gcd(y,\, \alpha_0 \beta_0) = 1,\, 
         \mu(y) \neq 0,\, \mu(\alpha \beta_0 y + \beta) \neq 0 \}|.
    \end{align*}
	
\begin{theorem} \label{2.3:lemma}
    Given the restrictions on the variables $\rho, \gamma, \alpha, \ao, \beta $ and $\bo$ as in \eqref{6.1:eqn}, we have
    \begin{align*}
        U(X) = X \left( \frac{\phi(\ao \bo)}{\ao \bo \gamma^2 \zeta(2)} \right) \prod_{p \mid \ao \bo \gamma} \left( 1- \frac{1}{p^2} \right)^{-1} \prod_{p \nmid \alpha \beta \gamma} \left( 1- \frac{1}{p^2-1} \right) + O(X^{3/4})
    \end{align*} 
    where the implied constant is dependent on $\gamma, \alpha,\ao,\b$ and $\bo$.
\end{theorem}

{}
%where $\ell$ is a prime number not dividing $km$
\begin{proof}[Proof of Theorem~\ref{1.1:thm}]
    Further, for each prime $p$ dividing $m$, select $\lambda_p \in (\Z/p\Z)^*$ such that $\lambda_p^p \not\equiv \lambda_p \pmod{p^2}$. Now consider the following system of congruences
    \begin{align} \label{6.3:eqn}
        x \equiv q \Mod{q^2} \quad \text{ and } \quad x \equiv \lambda_p \pmod{p^2},~ \forall\, p \mid m 
    \end{align}
    Then by Chinese remainder theorem, there exists a solution $\lambda$ to \eqref{6.3:eqn} such that $\lambda \equiv \lambda_p \pmod{\gamma^2}$ for all prime $p \mid m$, where $\gamma = q\rad(m)$. As $\lambda_p$ are coprime to $p$ dividing $\gamma$. Note that we need $q \mid a$, let us split $a$ as $a= q a_1$. Let $\rho, \gamma, \a,\b,\ao,\bo$ and $y$ be the integers defined as
    $$ \rho=\lambda, ~ \gamma= q\rad(m),  ~ y=a_1, ~ \a_0=1, ~ \bo=1, ~ \beta=m^m,  ~\a= q\sfrak^{\sfrak} (m-\sfrak)^{m-\sfrak}. $$
    
    Using $\gcd(m,q\sfrak)=1$ and the definition of $\a$, it is easy to observe that $\gcd(\a,\b)=1$. Further, for prime $p$ dividing $\gamma$, we have $p=q$ or $p \mid m$ and $p \nmid \lambda$. If $p=q$, then $p^2 \mid q\rho $ and $p \nmid m^m$. If $p\neq q$, then using $p \mid m$ along with $\gcd(m,\sfrak)=1$, we get $p^2 \mid m^m$ and $p \nmid \rho\sfrak(m-\sfrak)$. Thus we have $\a\bo y + \b \not\equiv 0 \Mod{p^2}$ for all primes $p$ dividing $\gamma$. Therefore, the variables $\rho, \a,\b,\ao,\b$ and $\gamma$ satisfy the conditions of \eqref{6.1:eqn}.
    
    We wish to estimate the following
    $$ U(X) := | \{ a \leq X : a \equiv \rho \Mod{\gamma^2}, ~\mu(a) \neq 0, ~\mu(\a\bo a+\b) \neq 0 \} |. $$
    Applying Theorem \ref{2.3:lemma}, we obtain
    \begin{align*}
        U(X) &=  \frac{X}{\gamma^2 \zeta(2)} \prod_{p \mid qm} \left( 1- \frac{1}{p^2}\right)^{-1} \prod_{p \nmid qm\sfrak (m-\sfrak)} \left( 1-\frac{1}{p^2-1} \right) + O(X^{3/4}),
    \end{align*}
    where the implied constant depends on $q,m$ and $\sfrak$. For the choices of $a$ obtained above, we construct the polynomial $f(x) = (x^{q^j} \pm 1)^m - a x^{n}$, where $j =v_q(n)$. The polynomial $f(x)$ defined above are monogenic by Corollary \ref{6.1lemma}.
\end{proof}

For the proof of Theorem \ref{2.10:thm}, we require a result which ensures that for any given polynomial $F(x)$, there are infinitely many primes $p$ such that $F(p)$ is squarefree. For the statement of this theorem, we will need the definition of local obstruction.

\begin{definition}
    Let $F(x) \in \Z[x]$ and $p$ be a prime. The polynomial $F(x)$ is said to have a local obstruction at $p$ if for all $z \in (\Z/p^2\Z)^*$, $F(z) \equiv 0 \pmod{p^2}$.
\end{definition}

The statement in Theorem \ref{2.5:cor} is a special case of a general result proved in \cite[Theorem 1.1]{Pasten2015}. It is a well known concept among the analytic number theorists. For a thorough explanation and discussion on this theorem, refer \cite[Remark 2.6]{L3} and the subsequent discussion therein.

\begin{theorem} \label{2.5:cor}
    Let $F(x) \in \Z[x]$, and suppose that $F(x)$ factors into product of distinct irreducibles, where the largest degree of any irreducible factor of $F(x)$ is $d$. Suppose further that, for each prime $q$, $F(x)$ does not have a local obstruction at $q$. If $d \leq 3$, or if $d \geq 4$ and assuming the $abc$-conjecture for number fields for $F(x)$, there exist infinitely many primes $p$ such that $F(p)$ is squarefree.
\end{theorem}

\begin{proof}[Proof of Theorem \ref{2.10:thm}]
    Consider the polynomial, 
    $$ F(x) = n^n (km-n)^{km-n} x^k - (km)^{km} \in \Z[x]. $$
    
    We wish to apply Theorem \ref{2.5:cor} to the polynomial $F(x)$. For this we first have to ensure that $F(x)$ has no repeated roots. Further, $F(x)$ should not have any local obstructions. That is, for each prime $\ell$, we must show that there exists some $z \in (\Z/\ell^2\Z)^*$ such that $F(z) \not\equiv 0 \Mod{\ell^2}$.
		
    One can easily verify that the only possible repeated root of $F(x)$ is $0$ which is clearly not a root of $F(x)$. Hence, $F(x)$ has no repeated factor. Now we proceed to prove that $F(x)$ has no local obstruction. Let $\ell$ be a prime and suppose that $ n(km-n) \equiv 0$ (mod $\ell$). This implies either $\ell \mid n$ or $\ell \mid (km-n)$. If $\ell \mid n$ (or $\ell \mid (km-n)$), then using $\gcd(km, n)=1$, we have $F(1)\not\equiv 0 \Mod{\ell^2}$. Further, if $\ell\mid km$, then using $\ell\nmid n(km-n)$, we obtain $F(1)\not\equiv 0 \Mod{\ell^2}$. Hence $F(x)$ does not have a local obstruction at primes $\ell$ dividing $kmn(km-n)$.
		
    Now assume that $\ell \nmid n(km-n)$. If $\ell$ is a prime dividing $km$, then we have $\ell^2$ divides $(km)^{km}$. Using this, we obtain $F(1) \equiv n^n (km-n)^{km-n}\Mod{\ell^2}$. Keeping in mind that $\gcd(n,km)=1$ and using $\ell \mid km$, we have $\ell$ does not divides $F(1)$. Therefore, we have $F(1) \not \equiv 0 \Mod{\ell^2}$, i.e., $F(x)$ does not have a local obstruction at primes $\ell$ dividing $km$ as well. %divide the right side of this congruence, hence $F(1) \not\equiv 0$ (mod $q^2$).
    
    Now assume that $\ell$ is a prime such that $\ell \nmid nkm(km-n)$. Consider,
    \begin{align*}
        F(1+\ell) - F(1) &= n^n (km-n)^{km-n} [ (1+\ell)^k - 1^t] \\
	&\equiv n^n (km-n)^{km-n} (\ell k)~ \Mod{\ell^2}.
    \end{align*}
    We use $\ell \nmid nk(km-n)$, to deduce that the right side of the above congruence is not divisible by $\ell^2$ in this case. This implies $F(1+\ell) \not\equiv F(1)$ (mod $\ell^2$). Hence, we obtain that either $F(1+\ell)$ or $F(1)$ is not congruent to $0$ modulo $\ell^2$. This concludes the proof that $F(x)$ has no local obstruction at any prime $\ell$.
		
    Thus, by Theorem \ref{2.5:cor} there exists infinitely many primes $p$ such that $F(p)$ is squarefree, where the existence of primes $p$ is unconditional when $k=2,3$ and conditional on the $abc$-conjecture for number fields when $k \geq 4$. For any of these primes $p$, we have
    \begin{align*}
        -(\calC-\calE) &= n^n (km-n)^{km-n} p^k - (km)^{km} = F(p),
    \end{align*}
    is squarefree. Hence, by Corollary \ref{6.1lemma}, $f(x)= (x^k \pm 1)^m - px^n$ is monogenic. This completes the proof of Theorem \ref{2.10:thm}.
\end{proof}

\section{Non Monogeneity of \texorpdfstring{$f(x)$}{f(x)}}
We Start this section introduction to the notation and reviews background material on Newton polygons. Let $p$ be a fixed rational prime. We write $\F_p$ for the finite field with $p$ elements,
$v_p$ for the usual $p$-adic valuation on $\Q$, $\Q_p$ for its $p$-adic completion, and $\Z_p$
for the ring of $p$-adic integers.

To analyze polynomials over $\Q_p$, we recall the notion of the Gauss valuation on the field
of rational functions $\Q_p(x)$, which extends the $p$-adic valuation $v_p$ on $\Q_p$.

\begin{definition}\label{def:Gauss}
The \emph{Gauss valuation} on $\Q_p[x]$ is defined by
\[
v_{p,x}\!\left(a_0+a_1x+a_2x^2+\cdots+a_nx^n\right)
= \min\{\,v_p(a_i)\mid 0\le i\le n\,\},
\]
where $a_i\in\Q_p$.
It is extended to rational functions in $\Q_p(x)$ by
\[
v_p\!\left(\frac{P}{Q}\right)=v_p(P)-v_p(Q),
\qquad 0\ne P,Q\in\Q_p[x].
\]
\end{definition}

Let $\varphi\in\Z_p[x]$ be a monic polynomial whose reduction modulo $p$, denoted by $\bar\varphi(x)$,
is irreducible in $\F_p[x]$. Furthermore, assume that $\varphi$ lifts to a monic irreducible factor
of $g(x)$ modulo $p$. Let $\F_\varphi$ be the residue field
\[
\F_\varphi=\F_p[x]/(\bar\varphi(x))\simeq\Z_p[x]/(p,\varphi(x)),
\]
(note that if $\deg(\varphi)=1$, then $\F_\varphi\simeq\F_p$).
For $g(x)\in\Z_p[x]$, there is a unique $\varphi$-expansion
\[
g(x)=\sum_{i=0}^t a_i(x)\varphi(x)^i,
\qquad \deg a_i(x)<\deg\varphi.
\]
This expansion is obtained by successive divisions of $g(x)$ by powers of $\varphi(x)$. The \emph{$\varphi$-Newton polygon} $N_\varphi(g)$ of a polynomial $g(x)$ for $p$ is defined as
the lower convex hull of the set of points $\{\, (i, v_p(a_i(x))) \mid a_i(x)\ne 0 \,\}$ in the Euclidean plane, where the $a_i(x)$ are the coefficients in the $\varphi$-expansion of $g(x)$. Geometrically, $N_\varphi(g)$ consists of a finite sequence of line segments $S_1,\dots,S_r$, called \emph{sides}, ordered by increasing slope. The polygon determined by the sides of negative slopes of $N_\varphi(g)$ is called the \emph{principal $\varphi$-Newton polygon} of $g(x)$ and is denoted by $N_\varphi^+(g)$. The length of $N_\varphi^+(g)$ is $\ell\bigl(N_\varphi^+(g)\bigr)=v_\varphi(\bar g(x)),$ the highest power of $\varphi$ dividing $g(x)$ modulo $p$. Let $S$ be a side of $N_\varphi^+(g)$ with initial point $(s,v_p(a_s(x)))$ and slope $\lambda=-h/e$, where $h$ and $e$ are positive coprime integers. The length $\ell$ of $S$ is the length of its projection onto the horizontal axis, and its degree, defined as $d=\ell/e$, is the number of segments into which the integral lattice divides $S$. We define the \emph{residual polynomial} attached to the side $S$ as $R_\lambda(g)(y)=c_0+c_ey+\cdots+c_{(d-1)e}y^{d-1}+c_{de}y^d\in\F_\varphi[y],$ where the residue coefficients $c_i\in\F_\varphi$ are given by
\[
c_i=
\begin{cases}
0, & \text{if } (s+i,\,v_p(a_{s+i}(x))) \text{ lies strictly above } S,\\[0.4em]
\dfrac{a_{s+i}(x)}{p^{v_p(a_{s+i}(x))}} \bmod (p,\,\varphi(x)),
& \text{if } (s+i,\,v_p(a_{s+i}(x))) \text{ lies on } S.
\end{cases}
\]

Note that for integers $i$ not divisible by $e$, the points $(s+i,\, v_p(a_{s+i}(x)))$ do not lie on $S$,
and hence the corresponding residue coefficients $c_i$ are zero.
The \emph{$\varphi$-index} $\it{ind}_\varphi(g)$ of $g(x)$ is $\deg(\varphi)$ multiplied by the number
of lattice points with positive integer coordinates that lie on or strictly below the principal
$\varphi$-Newton polygon $N_\varphi^+(g)$ (see \cite[Definition~1.3]{EFMN12}).

We say that a polynomial $g(x)\in\Z_p[x]$ is \emph{$\varphi$-regular} if all residual polynomials
$R_\lambda(g)(y)$ associated with the sides $S\subset N_\varphi^+(g)$ are square-free in
$\F_\varphi[y]$. If this condition holds for every $\varphi$ appearing in the factorization
of $g(x)\pmod p$, then $g$ is said to be \emph{$p$-regular}.
Let $\bar g(x)=\prod_{i=1}^r \bar\varphi_i^{\,\ell_i}$
be the factorization of $g(x)\pmod p$ into powers of monic irreducible coprime polynomials
over $\F_p$. Then $g(x)$ is called $p$-regular if it is $\varphi_i$-regular with respect to $p$
for every $i=1,\dots,r$.
For each $i$, let $N_{\varphi_i}^+(g)$ denote the principal Newton polygon of $g(x)$ with respect
to $\varphi_i$, having sides $S_{i,1},\dots,S_{i,r_i}$.
Let the residual polynomial corresponding to the side $S_{i,j}$ be denoted by
$R_{\lambda_{i,j}}(g)(y)$, and suppose it factors in $\F_{\varphi_i}[y]$ as
\[
R_{\lambda_{i,j}}(g)(y)=\prod_{s=1}^{s_{i,j}} \psi_{i,j,s}(y)^{a_{i,j,s}}.
\]

With this setup, we state the index theorem of Ore.

\begin{theorem}[Ore]\label{thm:Ore}
{\cite[Theorems~1.7 and~1.9]{EFMN12}}
Under the above hypotheses, the following hold:
\begin{enumerate}
\item
\[
v_p\bigl([\Z_K:\Z[\alpha]]\bigr)\ge
\sum_{i=1}^r \it{ind}_{\varphi_i}(g).
\]
The equality holds if $g(x)$ is $p$-regular.

\item
If $g(x)$ is $p$-regular, then
\[
p\Z_K=\prod_{i=1}^r \prod_{j=1}^{r_i} \prod_{s=1}^{s_{i,j}} \mathfrak{p}_{i,j,s}^{\,e_{i,j}}
\]
is the factorization of $p\Z_K$ into powers of prime ideals of $\Z_K$ lying above $p$, where
$e_{i,j}=\ell_{i,j}/d_{i,j}$ is the ramification index, $\ell_{i,j}$ is the length of $S_{i,j}$,
$d_{i,j}$ is the degree of $S_{i,j}$, and
\[
f_{i,j,s}=\deg(\varphi_i)\deg(\psi_{i,j,s})
\]
is the residue degree of the prime ideal $\mathfrak{p}_{i,j,s}$ over $p$.
\end{enumerate}
\end{theorem}

The following lemma (cf.~\cite{Hen94}), which characterizes the prime divisors of $i(K)$,
is crucial for the proof of the main theorems.

\begin{lemma}\label{lem:index}
Let $K$ be an algebraic number field and $p$ a rational prime.
Then $p$ divides $i(K)$ if and only if there exists a positive integer $f$ such that the number
of distinct prime ideals $\mathfrak{P}$ of $\Z_K$ lying above $p$ with residue degree $f$
is greater than the number of monic irreducible polynomials $N_f$ of degree $f$ over $\F_p[x]$.
\end{lemma}

\begin{proof}[Proof of Theorem~\ref{nonmonogeneity}]
By hypothesis, we have 
$f(x)\equiv x^{k+1}\bigl(x^{km-k-1}-1\bigr)\pmod{p}$. 
Write $km-k-1=r(p-1)$ for some integer $r$ with $p\nmid r$. 
Then $f(x)\equiv x^{k+1}\bigl(x^{r(p-1)}-1\bigr)\pmod{p}$. 
Since $p\nmid r$, the polynomial $x^{r(p-1)}-1$ is separable over $\mathbb{F}_p$ and hence factors as 
$x^{r(p-1)}-1\equiv \prod_{i=1}^{p-1}(x-i)\,U(x)\pmod{p}$, where $\gcd(x-i,U(x))=1$ for each $i$.

Set $\varphi_i=x-i$ for $i=0,1,\ldots,p-1$. 
For $i\neq 0$, each $\varphi_i$ yields a unique prime ideal $\mathfrak{p}_{i11}$ of $\Z_K$ lying above $p$ with residue degree $1$. 
The ideal corresponding to $U(x)$, denoted by $\mathfrak{b}$, is unramified and decomposes into prime ideals with residue degree $f>1$, satisfying $\mathcal{P}_f$\footnote{ $\mathcal{P}_f$ denotes the number of distinct prime ideals $\mathfrak{P}$ of $\Z_K$ lying above $p$ with residue degree $f$.} $<\mathcal{N}_f$ for all $f>1$.

It remains to analyze the factor $\varphi_0=x$ via the Newton polygon $N^+_{\varphi_0}(f)$. 
With respect to the $p$-adic valuation, the successive vertices of $N^+_{\varphi_0}(f)$ are 
$(km-k-1,0)$, $(km-k,\,v_p(m)+(m-1)v_p(c))$, and $(km,\,mv_p(c))$. 
Thus $N^+_{\varphi_0}(f)$ has two sides with slopes 
$\lambda_1=v_p(m)+(m-1)v_p(c)$ and $\lambda_2=(v_p(c)-v_p(m))/k$. 
The condition $\gcd(v_p(c)-v_p(m),k)=1$ implies that
$p\Z_K=\mathfrak{p}_{001}\mathfrak{p}_{011}^{\,k}\prod_{i=1}^{p-1}\mathfrak{p}_{i11}\mathfrak{b}$,
where all $f_{ijs}=1$. By Lemma~\ref{lem:index}, it follows that $p\mid i(K)$. Moreover, by \cite[Corollary, p.~230 and Theorem~4]{Engstrom1930}, we obtain $v_p(i(K))=1$, completing the proof.
\end{proof}

\section{Results on Galois Group}
To study the Galois group of the polynomial \( f(x) \), we make essential use of the following result due to Hajir, which provides a powerful criterion for determining when the Galois group is large.
\begin{lemma}[{\cite[Theorem~2.2]{Hajir2005Laguerre}}]\label{hajir}
Let $f(x)$ be an irreducible polynomial of degree $m$ and suppose $q$ is prime in the interval $(m/2,\,m-2)$ such that the Newton polygon with respect to some prime $p$ has an edge with slope $a/b$ where $a$ and $b$ are relatively prime integers and $q \mid b.$ Let $\D_f$ be the discriminant of $f(x).$ Then the Galois group of $f(x)$ over $\Q$ is the alternating group $A_m$ if $\D_f$ is a square and is the symmetric group $S_m$ if $\D_f$ is not a square
\end{lemma}

\begin{lemma}\label{thm:setup-galois} Let $f(x) = (x^{k}+c)^{m} - a x^{n} \in \mathbb{Z}[x]$ be a polynomial over $\mathbb{Q}$, where $k,m,n \in \mathbb{N}$ are such that $km-n$ is a prime number. Let $p$ be a prime such that $p \mid c$, $p \mid a$ and $p^{2} \nmid a$. Then the Newton polygon of $f(x)$ with respect to $p$ has an edge of length $km-n$.
\end{lemma}
\begin{proof}

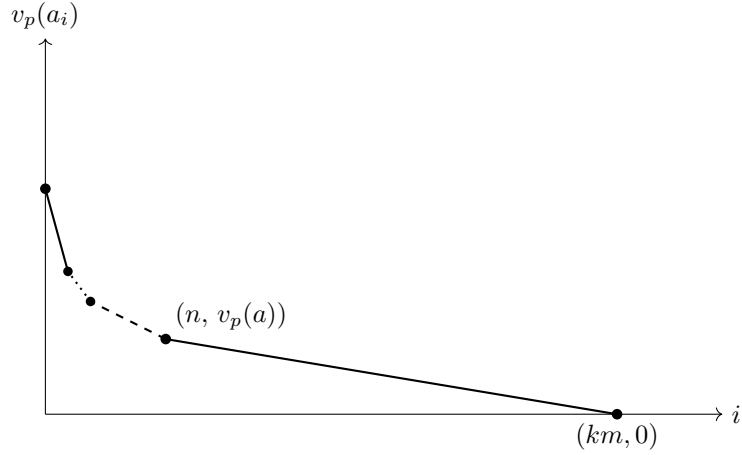
\begin{figure}[htb!]
\centering
\begin{tikzpicture}[scale=1]

% Axes
\draw[->] (0,0) -- (9,0) node[right] {$i$};
\draw[->] (0,0) -- (0,5) node[above] {$v_p(a_i)$};

% Vertices
\coordinate (O)  at (7.6,0); 7.6

% First edge endpoint (small slope)
\coordinate (B1) at (1.6,1);  7.3    % slope = 1

% Dashed intermediate points (strictly increasing slopes)
\coordinate (B3) at (0.6,1.5);   7 
\coordinate (B4) at (0.3,1.9);  6

% Final vertex (short, steep edge)
\coordinate (F)  at (0,3);   0 

% First edge (solid)
\draw[thick] (O) -- (B1);

% Intermediate edges (dashed/dotted)
\draw[dashed, thick] (B1) -- (B3);
\draw[dotted, thick] (B3) -- (B4);

% Final edge (solid, short and steep)
\draw[thick] (B4) -- (F);

% Points
\fill (O) circle (2pt) node[below] {$(km,0)$};
\fill (B1) circle (2pt) node[above right] {$(n,\,v_p(a))$};
\fill (B3) circle (1.8pt);
\fill (B4) circle (1.8pt);
\fill (F) circle (2pt);

\end{tikzpicture}
\caption{Newton polygon of $f(x)$ with respect to $p$}
\label{fig:newton-polygon}
\end{figure}
Write $f(x)=(x^{k}+c)^{m}-a x^{n}=\sum_{i=0}^{km} a_i x^i,$
and consider the Newton polygon of $f(x)$ with respect to the prime $p$. Since $p\mid c$ and $p\mid a$ with $p^{2}\nmid a$, the $p$-adic valuations of the coefficients $a_i$ can be described explicitly. Expanding $(x^{k}+c)^m$, we see that the coefficient of $x^{ik}$ is $\binom{m}{i}c^{\,m-i}$ for $0\le i\le m$. In particular, the constant term $c^{m}$ satisfies $v_p(c^{m})\ge m$, while the leading coefficient is $1$.
Moreover, the coefficient of $x^{n}$ is $-a$, and by assumption $v_p(a)=1$. It follows that among all coefficients $a_i$ with $i<km$, the smallest
$p$-adic valuation occurs at $i=n$, and this valuation is exactly $1$. Consequently, the point $(n,1)$ lies strictly below all other points $(i,\,v_p(a_i))$ with $i<n$. On the other hand, the constant term corresponds to the point $(0,v_p(c^{m}))$, which lies above the line joining $(0,0)$ and $(n,1)$. Since $km-n$ is the prime, there are no other lattice points with a smaller $p$-adic valuation lying below the line segment joining $(n,1)$ and $(km,0)$. Thus, the lower convex hull of the points $(i,v_p(a_i))$ consists of a single
edge that joins $(n,1)$ to $(km,0)$. Equivalently, the Newton polygon of $f(x)$ with respect to $p$ has an single edge of horizontal length $km-n$. This situation is illustrated in Figure~\ref{fig:newton-polygon}, where the intermediate points are strictly above the unique lower edge. This completes the proof of the lemma.
\end{proof}

\begin{proposition}\label{cor:SnAn}
Under the assumptions of Proposition~\ref{thm:setup-galois}, let $d=\deg(f)=km$. If $\frac{d}{2} < km-n < d-2,$
then the Galois group of $f(x)$ over $\mathbb{Q}$ contains $A_{d}$. Moreover, the Galois group is $A_{d}$ if $\D_f$ is a square in $\mathbb{Q}$ and $S_{d}$ otherwise.
\end{proposition}
\begin{proof}
    The proposition is an easy consequence of Lemmas \ref{hajir} and \ref{thm:setup-galois}.
\end{proof}
\begin{proof}[Proof of Theorem~\ref{galois group}]
Let $d=km=\deg(f)$. By assumption, $q=km-n$ is a prime satisfying $\frac{d}{2} < q < d-2.$ Hence, by Proposition~\ref{cor:SnAn}, the Galois group $G=\mathrm{Gal}(K/\mathbb{Q})$ contains the alternating group $A_d$.
Therefore, it remains to determine whether $G=A_d$ or $G=S_d$.
By Proposition~\ref{cor:SnAn}, this depends on whether the discriminant
$\D_f$ is a square in $\mathbb{Q}$. Recall that the discriminant of $f(x)$ is
\[
\D_f
=\pm
a^{\,k(m-1)}c^{m(n-1)}\,
\left[\,(km)^{\,k_1m}c^{k_1m-n_1}
      - a^{k_1}n^{\,n_1}(km-n)^{\,k_1m-n_1}
\right]^{t},
\]
where $t=\gcd(n,k)$, $n=n_1t$, and $k=k_1t$. We now show that $\D_f$ is not a square in $\mathbb{Q}$. By assumption, there exists a prime $\ell$ such that $v_\ell(a) \text{ is odd } \text{ and }
\ell\nmid kmc.$ Since $\ell\nmid kmc$, the prime $\ell$ does not divide $(km)^{k_1m}$ and $
c^{k_1m-n_1}$. Hence $\ell$ does not divide $\calC-\calE$. Therefore, $v_\ell(\D_f) = k(m-1)\,v_\ell(a).$ By hypothesis, $k(m-1)$ is odd and $v_\ell(a)$ is odd. Thus $v_\ell(\D_f)$ is odd. Hence $\D_f$ is not a square in $\mathbb{Q}$. Therefore, $G=S_d.$ This completes the proof.
\end{proof}

\section{An application to Differential Equations}
\begin{proof}[Proof of Theorem~\ref{t2}]
The given differential equation is \begin{align}{\label{diff1}}
    \left(\frac{d^k }{dx^k}+c\right)^my-a\frac{d^ny }{dx^n}=0 
    \end{align}
    where $km>n\geq 1$.  Observe that $\mathcal{F}(z) = (z+c)^m-az^n$ is the irreducible auxiliary equation associated with \eqref{diff1}, having root $\theta$. Suppose that every prime $p$ dividing the discriminant $\D_{\mathcal{F}}$ satisfies one of the conditions (i)–(vii) of Theorem \ref{monogeneity of f(x)}. Then, by the formula $
\D_{\mathcal{F}} = [\Z_K : \mathbb{Z}[\theta]]^2 d_K,
$ it follows that $\Z_K = \mathbb{Z}[\theta]$, where $\Z_K$ denotes the ring of integers of the number field $K = \mathbb{Q}(\theta)$. Also
 $$
 \mathbb{Z}[\theta] = \{c_0 + c_1\theta + c_2 \theta^2 +\cdots + c_{km-1}\theta^{km-1} \;|\; c_k \in \mathbb{Z} \text{ for all } 1\leq k\leq km-1     \}.
 $$
Thus, every root of the equation $\mathcal{F}(z) = 0$ can be expressed in the form
$$
c_0^{(i)} + c_1^{(i)}\theta + c_2^{(i)}\theta^2 + \cdots + c_{km-1}^{(i)}\theta^{km-1},
$$
where each $c_{j-1}^{(i)} \in \mathbb{Z}$ for all $1 \leq i,j \leq km$. Consequently, the general solution of the differential equation \eqref{diff1} is given by
\begin{align*}
    y(x) = \sum_{i=1}^{km} \alpha_i\prod_{j=1}^{km}e^{c_{j-1}^{(i)} \theta^{j-1}x},
\end{align*}
where  $\alpha_i$ is arbitrary real constants for all $1\leq i,j \leq km$.This establishes the claimed form of the general solution and completes the proof.
\end{proof}
\begin{remark}
 The key idea of the proof is to relate the analytic structure of the differential equation to the arithmetic structure of its auxiliary equation. Once the associated polynomial is monogenic, its roots can be described
explicitly in terms of an integral basis, allowing a uniform description of all exponential solutions.
\end{remark}

\section{Examples}
This section is devoted to examples that illustrate our main results.
\begin{example}
Let $q$ be a prime such that $q \mid\mid a$, and let $m \ge 2$ be an integer. Let $f(x) = (x^{q} + 1)^m - ax \in \mathbb{Z}[x]$. Then, by Corollary~\ref{6.1lemma}, the polynomial $f(x)$ is irreducible. Moreover, suppose $p$ is a prime such that $p \mid \D_f$; then $f(x)$ is monogenic if and only if the following conditions hold:
\begin{enumerate}
    \item $a$ is square-free;
    \item If $p \nmid a$ and $p \mid m$, then $p^2 \nmid (a^p - a)$;
    \item If $p \nmid am$, then $p^2 \nmid (\mathcal{C} - \mathcal{E})$.
\end{enumerate}

From Theorem~\ref{disc of f(x)}, we have \[\D_f = \pm a^{q(m-1)} \left[ (qm)^{qm} - a^q (qm-1)^{qm-1} \right] =a^{q(m-1)}N.\] To generate families of monogenic polynomials, we fix certain values of $a$ and $m$ and compute the corresponding factorization of $N$, which are presented in the tables below.

\begin{table}[h]
\centering

\begin{minipage}{0.48\textwidth}
\centering
\begin{tabular}{|c|c|l|}
\hline
$a$ & $m$ & Factorization of $|N|$ \\
\hline
2 & 2 & $2^2 \cdot 37$ \\
2 & 3 & $2^2 \cdot 8539$ \\
2 & 5 & $2^2 \cdot 7 \cdot 17 \cdot 61 \cdot 71 \cdot 4099$ \\
\hline

6 & 2 & $ 2^2 \cdot 179$ \\
6 & 4 & $ 2^2 \cdot 31 \cdot 271 \cdot 383$ \\
6 & 5 & $ 2^2 \cdot 9049 \cdot 109049$ \\
\hline
\end{tabular}
\end{minipage}
\hfill
\begin{minipage}{0.48\textwidth}
\centering
\begin{tabular}{|c|c|l|}
\hline
$a$ & $m$ & Factorization of $|N|$ \\
\hline

10 & 2 & $ 2^2 \cdot 13 \cdot 47$ \\
10 & 3 & $ 2^2 \cdot 41 \cdot 1621$ \\
10 & 4 & $ 2^2 \cdot 3 \cdot 59 \cdot 92623$ \\
\hline

14 & 2 & $ 2^2 \cdot 1259$ \\
14 & 3 & $ 2^2 \cdot 141461$ \\
14 & 4 & $ 2^2 \cdot 3 \cdot 53 \cdot 277 \cdot 821$ \\
14 & 5 & $2^2 \cdot 11 \cdot 41 \cdot 43 \cdot 397 \cdot 2141$ \\
\hline
\end{tabular}
\end{minipage}
\\[2mm]
\caption{Factorization of $N$ for $q=2$ and various values of $a$ and $m$.}
\end{table}
First, we analyze the contribution of the factor $a$ to the discriminant 
$\D_f$ of $f(x)$. From the tables above, for each choice of $a$ we observe that whenever $p \mid a$, we have $p^2 \nmid a$. Hence, by Condition~(1), it follows that 
$p \nmid [\mathbb{Z}_K : \mathbb{Z}[\theta]]$.

Next, we consider the factorization of $N$. For the given values of $a$ and $m$ appearing in the tables, one checks that $N$ has no repeated prime factors other than $2$. Since $2 \mid  a$ and $2^2 \nmid a$, Condition~(1) implies 
$2 \nmid [\mathbb{Z}_K : \mathbb{Z}[\theta]]$.

Now, let $p$ be any odd prime divisor of $N$. From the tables, we see that 
$p^2 \nmid N$. Therefore, by Condition~(3), we conclude that 
$p \nmid [\mathbb{Z}_K : \mathbb{Z}[\theta]]$.

Combining the above arguments, we conclude that for each pair of values $(a,m)$ listed in the tables and $q=2$, the corresponding polynomial is monogenic.

By applying a similar argument and varying the values of $q$, $a$, and $m$ in such a way that $N$ has no repeated prime factors other than $q$, many further examples of monogenic polynomials can be obtained that satisfy the above conditions.
\end{example}

\begin{example}
Let $p$ be a prime number and let $k,n \ge 1$ be integers. 
Choose $a$ of the form $a = 2^p + pt,$ where $t \in \mathbb{Z}$ and $p \nmid t$. Consider the polynomial  \[f(x) = (x^k+1)^p - a x^n \in \mathbb{Z}[x].\] Set $g(x)=f(x+1)$. Then $g(x)=((x+1)^k+1)^p - a(x+1)^n.$
Since $p$ is prime, the Binomial identity implies that $(u+v)^p \equiv u^p+v^p \pmod p.$ Therefore, all intermediate binomial coefficients in the expansion of  $((x+1)^k+1)^p$ are divisible by $p$. It follows that every coefficient of  $g(x)$, except possibly the leading coefficient and the constant term,  is divisible by $p$. The constant term of $g(x)$ is $2^p - a = 2^p - (2^p+pt) = -pt.$ Thus $p \mid (2^p-a)$, but since $p \nmid t$, we have  $p^2 \nmid (2^p-a)$. Therefore, $g(x)$ satisfies the Eisenstein criterion with respect to the prime $p$. Hence $g(x)$ is irreducible over $\mathbb{Q}$, and consequently $f(x)$ is irreducible over $\mathbb{Q}$. By Theorem~\ref{disc of f(x)}, we  have \[\D_f=\pm a^{k(p-1)}\,[(pk)^{pk_1}-a^{k_1}n^{n_1}(pk-n)^{pk_1-n_1}]^t,\] where $t=\gcd(n,\,k)$
\medskip
\begin{table}[H]
\centering
\begin{tabular}{|c|c|c|c|l|}
\hline
$p$ & $a$ & $n$ & $k$ & Factorization of $|N|$ \\
\hline
3 & 11  & 5 & 2 & $109 \cdot 3041$ \\
3 & 14  & 2 & 3 & $23 \cdot 1439 \cdot 261407$ \\
3 & 23  & 5 & 2 & $19 \cdot 84551$ \\
7 & 149 & 3 & 2 & $29 \cdot 691 \cdot 3061 \cdot 4447 \cdot 586237$ \\
\hline
\end{tabular}
\caption{Square-free factorizations of $N$ for selected values of $p$, $a$, $n$, and $k$.}
\label{tab:squarefreeN}
\end{table}
\medskip

First, we analyze the contribution of the factor $a$ to the discriminant  $\D_f$ of $f(x)$. From Table~\ref{tab:squarefreeN}, for each choice of $a$, we observe that whenever $p \mid a$, we have $p^2 \nmid a$. Hence, by Theorem~\ref{monogeneity of f(x)} (i), it follows that $p \nmid [\mathbb{Z}_K : \mathbb{Z}[\theta]]$. Now, let $p$ be any odd prime divisor of $N$. From  Table~\ref{tab:squarefreeN}, we see that  $p^2 \nmid N$. Therefore, by Theorem~\ref{monogeneity of f(x)} (vii), we conclude that $p \nmid [\mathbb{Z}_K : \mathbb{Z}[\theta]]$.
Combining the above arguments, we conclude that for each value $(p,a,n,k)$ listed in  Table~\ref{tab:squarefreeN}, the corresponding polynomial is monogenic.
\end{example}
The computations to obtain the factorization  in the above examples were performed using \texttt{SageMath}.
\begin{example}
Consider the polynomial
\[
f(x)=(x^{5}+3)^{6}-21x^{11}.
\]
Here $k=5$, $m=6$, and $n=11$, so $d=km=30$ and $q=km-n=30-11=19,$ which is prime. Moreover, $\frac{d}{2}=15 < 19 < 28=d-2.$ Hence, by Proposition~\ref{thm:setup-galois}, the Galois group of $f(x)$ contains $A_{30}$. We now verify the remaining hypotheses of Theorem~\ref{galois group}. First, $k(m-1)=5\cdot5=25$ is odd. Next, let $p=3$. Then $3\mid c$, $3\mid a=21$, and $3^{2}\nmid 21$, so condition (2) holds. Finally, take $\ell=7$. We have $v_7(a)=1$  and $7\nmid kmc = 5\cdot6\cdot3 = 90.$ Therefore,
$v_7(\D_f)=k(m-1)v_7(a)$ is odd, so the discriminant is not a square. Consequently, the Galois group of $f(x)$ over $\mathbb{Q}$
is the full symmetric group $S_{30}$.
\end{example}
\begin{example}
Consider the polynomial
$f(x)=(x+25)^6-6x^2$
and let $K=\mathbb{Q}(\theta)$, where $\theta$ is a root of $f(x)$. In this case, $p=5$, $k=1$, and $m=6$, so that
$km-k-1=6-1-1=4.$ Hence $p-1=4$ divides $(km-k-1)$ and $p\nmid (km-k-1)$. Moreover, $mk=6>k+1=2$. Here $a=6$ and $c=25$. We observe that $a\equiv 1 \pmod{5}$ and $5\mid c$.
Furthermore, $v_5(c)=2$ and $v_5(m)=0,$ so that $v_5(c)>(k+1)v_5(m)=2\cdot 0=0.$ Also, $\gcd\! \left (v_5(c)-v_5(m),k\right) = \gcd (2,1) = 1.$ Therefore, all the hypotheses of Theorem~\ref{nonmonogeneity} are satisfied. Consequently, $v_5\!\left(i(K)\right)=1.$
\end{example}

\bibliographystyle{alpha}
\bibliography{references}
\end{document}